\documentclass[a4paper,9pt]{amsart}

\usepackage{tikz}

\usepackage{pgflibraryarrows}
\usepackage{pgflibrarysnakes}
\usepackage{amsmath}
\usepackage{amssymb}
\usepackage{amsthm}
\usepackage[all]{xy}
\usepackage[latin1]{inputenc}        
\usepackage{amscd}
\usepackage{mathrsfs}
\usepackage[mathcal]{eucal}
\usepackage{float}
\author{Cinzia Bisi, Giampiero Chiaselotti}
\address{Bisi Cinzia\\Dipartimento di Matematica, Universit\'a di Ferrara, Via Machiavelli 35, 44121, Ferrara, Italy}
\email{bsicnz@unife.it}
\email{bisi@math.unifi.it}
\address{Chiaselotti Giampiero\\Dipartimento di Matematica, Universit\'a della Calabria, Via Pietro Bucci,
Cubo 30B, 87036 Arcavacata di Rende (CS), Italy.}
\email{chiaselotti@unical.it}
\title[Extension results for boolean maps]{Extension results for boolean maps and a class of systems of linear inequalities}
\date{\today}

\newtheorem{inizio}{Lemma}[section]
\newtheorem{theorem}[inizio]{Theorem}
\newtheorem{corollary}[inizio]{Corollary}
\newtheorem{proposition}[inizio]{Proposition}

\newtheorem{definition}[inizio]{Definition}

\newtheorem{o-problem}[inizio]{Open Problem}

\newtheorem*{teo-L}{Theorem}

\newtheorem*{corollary-s}{Corollary}
\theoremstyle{definition}
\newtheorem{example}[inizio]{Example}

\begin{document}


\abstract
In this paper we introduce the notion of {\it core} for two specific classes of boolean maps on finite involution posets (which are a generalization of the boolean lattices) and we prove some extension results for such families of boolean maps. Through the properties of the core, we provide a complete characterization of such maps. The main purpose of such abstract results is their application to the study of the compatibility of a particular class of systems of linear inequalities related to a conjecture of Manickam, Mikl\"os and Singhi (\cite{ManSin88}, \cite{ManMik87}), still unsolved and that can be considered dual to the theorem of Erd\"os-Ko-Rado \cite{erd-ko-rad}.
\endabstract

\maketitle
\section{Introduction}
In the sequel $n$ and $r$ will denote two fixed integers such that $1 \le r \le n$.\\
 We set $I(n,r) =\{ \tilde{r}, \cdots, \tilde{1}, \overline{1}, \cdots, \overline{n-r}\}$
  and we consider $I(n,r)$ as a $n$-set in which we simply have marked the difference between the first $r$ formal symbols $\tilde{r}, \cdots, \tilde{1}$ and the remaining $(n-r)$ formal symbols $\overline{1}, \cdots, \overline{n-r}$ . Let us suppose that we have $r$ real variables $x_{\tilde{r}},\cdots, x_{\tilde{1}}$  and other $(n-r)$ real variables $y_{\overline{1}},\cdots, y_{\overline{n-r}}$. In the sequel, to simplify the notations, we write simply $x_i$ instead of $x_{\tilde{i}}$ and $y_j$ instead of $y_{\overline{j}}$. However, it is important to mark the fact that the index $i$ in $x_i$ correspond to the symbol $\tilde{i}$, while the index $i$ in $y_i$ correspond to the symbol $\overline{i}$, and that $\tilde{i} \neq \overline{i}$.
We call $(n,r)$-{\it system} of {\it size} $p$ a system $\mathcal{S}$ of linear inequalities having the following form:

\begin{equation}\label{(n,r)-system}
\mathcal{S} :  \left\{ \begin{array}{l}
                    x_r \ge \cdots \ge x_1 \ge 0 > y_1 \ge \cdots \ge y_{n-r} \\
                    \sum_{i \in A_1} t_i \ge 0  \,\,\,(\textrm{or} \,\,\,  < 0)\\
                    \sum_{i \in A_2} t_i  \ge 0  \,\,\,(\textrm{or} \,\,\,  < 0)\\
                       \cdots  \\
                       \cdots  \\
                    \sum_{i \in A_p} t_i  \ge 0  \,\,\,(\textrm{or} \,\,\,  < 0)\\
                    \end{array}
  \right.
\end{equation}

where $A_1, \cdots, A_p$ are non-empty and different subsets of $I(n,r),$ all different from the singletons of $I(n,r);$ moreover $t_i=x_i$ if $i\in \{ \tilde{r}, \cdots, \tilde{1}\}$ and $t_i=y_i$ if $i\in \{ \overline{1}, \cdots, \overline{n-r}\}$. Formally we set $\sum_{i \in \emptyset} t_i = 0$. When the subsets $A_1, \cdots, A_p$ coincide with all the possible subsets of $I(n,r)$ different from the singletons and from the empty set, we say that the $(n,r)$-system (\ref{(n,r)-system}) is {\it total}. Furthermore, when in (\ref{(n,r)-system}) appears the inequality

\begin{equation} \label{DisegPeso}
 x_r + \cdots +  x_1 + y_1 + \cdots + y_{n-r} \ge 0 \,\,\,(\textrm{or} \,\,\,  < 0)
 \end{equation}

 we say that it is a $(n,r)$-{\it positively weighted} system (or a $(n,r)$-{\it negatively weighted} system).

 A $(n,r)$-positively [negatively] total weighted system of type (\ref{(n,r)-system}) can be identified with a boolean total map $A$ defined on the lattice $(S(n,r), \sqsubseteq )$. The lattice $(S(n,r), \sqsubseteq )$ has been introduced and studied in \cite{BisChias}. A subsystem of (\ref{(n,r)-system}) which is also equivalent to it, can be identified  with a particular restriction of the map $A$ which represents (\ref{(n,r)-system}): such a restriction of $A$ is called {\it core} of $A$.

 In this paper we introduce and study two particular classes of boolean maps on involution posets (a more general version of the boolean lattices). In particular, we shall prove some extension results for such maps and we shall apply these results to the study of the compatibility and solvability of the total $(n,r)$-systems. The classical approaches to the study of the linear inequalities systems usually make use of alternative theorems and operational research's methods (see for example \cite{Goffin80},  \cite{PangEtAltri2007}, \cite{Spingarn85}); our approach to the study of the $(n,r)$-systems uses instead the properties of a particular class of boolean maps defined on some types of involution posets introduced in \cite{BisChias}.

 Our motivation to study the combinatorial properties of a $(n,r)$-system of type (\ref{(n,r)-system}) is an attempt to answer to a conjecture of Manickam, Mikl\"os and Singhi (see \cite{ManMik87}, \cite{ManSin88}) that can be considered dual to the famous theorem of Erd\"os-Ko-Rado \cite{erd-ko-rad}, \cite{katona}. This conjecture is connected with the first distribution invariant of the Johnson association scheme (see \cite{bier-man}, \cite{manic-88}, \cite{manic-91}, \cite{ManSin88}). The distribution invariants were introduced by Bier \cite{bier}, and later investigated in \cite{bier-delsarte}, \cite{manic-86}, \cite{manic-88}, \cite{ManSin88}. Also, as pointed out in \cite{sriniv-98}, this conjecture settles some cases of another conjecture on multiplicative functions by Alladi, Erd\"os and Vaaler, \cite{erdos}. Partial results related to the Manickam-Mikl\"os-Singhi conjecture have been obtained also in \cite{bhatta-2003}, \cite{bhatt-thesis}, \cite{Chias02}, \cite{ChiasInfMar08}, \cite{chias-mar-2002}.

 Studying the Manickam-Mikl\"os-Singhi conjecture and some related extremal sum problems (see \cite{BisChias}, \cite{ManMik87}, \cite{ManSin88}), it can be necessary to determine an equivalent subsystem of (\ref{(n,r)-system}) which has the minimal possible number of inequalities. Such a subsystem can be identified with a core of $A$ having minimal cardinality. In this paper we show that if $A$ is the particular boolean map which represents (\ref{(n,r)-system}), then $A$ has exactly a unique core of minimal cardinality (we call it the {\it fundamental core} of $A$) and we also show as such core is made.
The utility of the core and of its properties will appear clear in the forthcoming paper \cite{chias-mar-nardi}.
In determining our results, we see that the essential property of the lattice $(S(n,r), \sqsubseteq )$ is the existence in the lattice $(S(n,r), \sqsubseteq )$ of a unary operation $c$ (the complement function) such that $i)$ $c(c(w)) = w$ for each $w \in S(n,r)$; $ii)$ $c(v) \sqsubseteq c(w)$ whenever $w \sqsubseteq v$; $iii)$ $c(v) \neq v$ for each $v$. The first two properties of $c$ are those that define an involution poset: recent studies related to this particular class of posets can be find in \cite{agha-gree-2001} and in \cite{brenneman-altri}. The properties $i)$,$ii)$ and $iii)$ define what we call a {\it strongly} involution poset (SIP). Therefore the natural context in which studying our problem in its most general form is the context of the finite strongly involution posets, that are also a more general version of the finite boolean algebras.

 Let us conclude this introduction observing that in this paper we carry out the research project started in \cite{BisChias}, which consists in the study of the extremal sum problems settled in \cite{ManMik87} and in \cite{ManSin88} (among which the Manickam-Mikl\"os-Singhi conjecture), in the setting of the order combinatorial theory. For further details we refer to \cite{BisChias}. In particular, our goal is to try to build a family of boolean maps on $S(n,r)$ that captures all the properties of the $(n,r)$-compatible total systems. This is important because the combinatorial properties of a $(n,r)$-system can be studied more easily if we see such system as a particular boolean map on $S(n,r)$. In this paper we define two families of boolean maps, each of one captures respectively some (but not all) combinatorial properties of a $(n,r)$-compatible positively [negatively] weighted system and we determine two families of subsets of $S(n,r)$ that are respectively in bijective correspondence with the previous families of boolean maps. With an image that recall the differential calculus, we can imagine these two families of boolean maps as a set of ``critical points", inside of which try to find the ``extreme points", i.e. those boolean maps that correspond to a compatible system. Therefore, the study of such families of boolean maps is important to delimit the research to the $(n,r)$-compatible systems.
 The problem that remain open is : what is a family of boolean maps on $S(n,r)$ that captures all the properties that characterize a $(n,r)$-compatible system?
 In the last section we suggest a family of boolean maps candidate for this job.

\section{$(n,r)$-systems vs boolean maps}

 Let $\mathcal{S}$ , $\mathcal{S'}$ be two $(n,r)$-systems: we say that they are {\it equals} (in symbols $\mathcal{S} = \mathcal{S'}$) if they have exactly the same inequalities, otherwise we say that they are {\it different} (in symbols $\mathcal{S} \neq \mathcal{S'}$). If they are both compatibles (i.e. they have solutions) and equivalents (i.e. they have the same solutions) we shall write $\mathcal{S}\equiv \mathcal{S'}$. We denote by $Syst(n,r)$ [$TSyst(n,r)$] the set of all the $(n,r)$-systems [$(n,r)$-total systems],  by $CSyst(n,r)$ [$CTSyst(n,r)$] the set of all the  $(n,r)$-systems that are also compatibles [totals and compatibles]; by $W_+Syst(n,r)$, $W_+CSyst(n,r)$, $W_+TSyst(n,r)$, $W_+CTSyst(n,r)$  we respectively denote the set of all the $(n,r)$-positively weighted systems, the $(n,r)$-compatible positively weighted systems, the $(n,r)$-total positively weighted systems, the $(n,r)$-compatible total positively weighted systems and by $W_-Syst(n,r)$, $W_-CSyst(n,r)$, $W_-TSyst(n,r)$, $W_-CTSyst(n,r)$ their analogue but negatively weighted.
 Let us note that if $\mathcal{S}, \mathcal{S'}\in CTSyst(n,r)$ and $\mathcal{S} \neq \mathcal{S'}$, then $\mathcal{S}$ and $\mathcal{S'}$ can not be equivalent.




 Now we briefly recall the definition of the lattice $S(n,r)$ that we have introduced in \cite{BisChias}.

 We set $A(n,r)=I(n,r)\cup \{0^§\}$, where $0^§$ is a new formal symbol. We introduce on $A(n,r)$ the following total order:
\begin{equation}\label{orderA(n,r)}
\overline{n-r} \prec \cdots \prec \overline{2} \prec \overline{1} \prec 0^§ \prec \tilde {1} \prec \tilde{2}\prec \cdots \prec \tilde{r},
\end{equation}
where $\overline{n-r}$ is the minimal element and $\tilde{r}$ is the maximal element in this chain.
If $i,j \in A(n,r),$ then we write $i \preceq j$ for $i=j$ or $i \prec j$.
 We denote by $(\mathcal{C}(n,r), \sqsubseteq)$ the $n$-fold cartesian product poset $A(n,r)^n$. An arbitrary element of $\mathcal{C}(n,r)$ can be identified with an $n$-string $t_1\cdots t_n$ on the alphabet $A(n,r)$. Therefore, if $t_1\cdots t_n$ and $s_1\cdots s_n$ are two strings of $\mathcal{C}(n,r)$, we have
 \begin{equation*}
 t_1\cdots t_n \sqsubseteq s_1\cdots s_n \Longleftrightarrow t_1 \preceq s_1, \cdots , t_n \preceq s_n.
\end{equation*}

We introduce now $S(n,r)$ as a particular subset of $\mathcal{C}(n,r)$.

A string of $S(n,r)$ is constructed as follows: it is a formal expression of the following type
\begin{equation} \label{stringa}
i_r \cdots  i_1 \,\,\, | \,\,\,  j_1 \cdots  j_{n-r},
\end{equation}
such that:

i) $i_1, \cdots, i_r \in \{ \tilde{1}, \cdots, \tilde{r}, 0^§ \},$

ii) $j_1, \cdots, j_{n-r} \in \{ \overline{1}, \cdots, \overline{n-r}, 0^§ \},$

iii) $i_r \succeq \cdots \succeq i_1 \succeq 0^§ \succeq j_1 \succeq \cdots \succeq j_{n-r},$

iv) the unique element in (\ref{stringa}) which can be repeated is $0^§$.

Then $S(n,r)$ is the subset of all strings of $\mathcal{C}(n,r)$ having the previous form with the induced order from $\sqsubseteq$.
The formal symbols which appear in (\ref{stringa}) will be written without  $\;\tilde{}\;$, $\;\bar{}\;$, and $\;^§\;$ and the vertical bar $|$ in (\ref{stringa}) will indicate that the symbols on the left of $|$ are in
$\{ \tilde{1}, \cdots, \tilde{r}, 0^§ \}$ and the symbols on the right of $|$ are elements in $\{ 0^§, \overline{1}, \cdots, \overline{n-r} \}.$
\begin{example}
If $n=3$ and $r=2,$ then $A(3,2)= \{ \tilde{2} \succ \tilde{1} \succ 0^§ \succ \overline{1} \}.$
Hence $S(3,2)= \{ 21|0, \,\,\, 21|1,\,\,\, 10|0,\,\,\, 20|0,\,\,\, 10|1,\,\,\, 20|1,\,\,\, 00|1,\,\,\, 00|0 \}.$\\
\end{example}

In \cite{BisChias} it has been proved that:

$i)$ $(S(n,r), \sqsubseteq)$ is a graded lattice with minimal element $0\cdots 0|12\cdots(n-r)$ and maximal element $r(r-1)\cdots21|0\cdots0$;

$ii)$ $(S(n,r), \sqsubseteq)$ has the following unary complementary operation $c$:

$(p_1\cdots p_k\,\,0\cdots0|0\cdots0\,\,q_1\cdots q_l)^c = p'_1\cdots p'_{r-k}\,\,0\cdots0|0\cdots0\,\,q'_1\cdots q'_{n-r-l}$,

where $\{p'_1,\cdots ,p'_{r-k}\}$ is the usual complement of $\{p_1,\cdots,p_k\}$ in $\{\tilde{1}, \cdots, \tilde{r}\}$, and\\
$\{q'_1,\cdots ,q'_{n-r-l}\}$ is the usual complement of $\{q_1,\cdots,q_l\}$ in $\{\overline{1}, \cdots, \overline{n-r}\}$ (for example, in $S(7,4)$, we have that $(4310|001)^c = 2000|023$).

Let us consider a $(n,r)$-system $\mathcal{S}$ as in (\ref{(n,r)-system}). Since there is an obvious bijection between the power set $\mathcal{P}(I(n,r))$ and $S(n,r)$, all the subsets $A_1,\cdots,A_p$ in (\ref{(n,r)-system}) can be identified with strings of $S(n,r)$, that we denote by $w_1,\cdots,w_p$ (for example, if $n=7, r=4$, we identify the subset $\{\tilde{1},\tilde{3},\tilde{4},\overline{1}\}$ with the string $4310|001$, or the subset $\{\tilde{2},\overline{2},\overline{3}\}$ with $2000|023$). Let us note that $0\dots0|0\cdots0$ will be identified always with the empty subset of $I(n,r)$).

By Proposition 6.1 of \cite{BisChias} it results that if $w_k \sqsubseteq w_j$ for some $k,j$, then $\sum_{i \in A_k} t_i \le \sum_{i \in A_j} t_i$.

We denote with $(S(n,r) \rightsquigarrow {\bf 2})$ the poset of the boolean partial maps on $S(n,r)$ (\cite{dav-pri-2002}).
We set now $\xi _r= r0 \cdots 0|0 \cdots 0$, $\cdots$, $\xi _1= 10 \cdots 0|0 \cdots 0$, $\xi _0= 00 \cdots 0|0 \cdots 0$, $\eta _1= 0 \cdots 0|0 \cdots01$, $\cdots$, $\eta _{n-r}= 0 \cdots 0|0 \cdots 0(n-r)$, and

$\Omega_{\mathcal{S}} = \{w_1,\cdots,w_p, \xi _r,\cdots,\xi _1,\xi _0,\eta_1,\cdots, \eta_{n-r}\}$.

\begin{definition}
Let $\mathcal{S} \in Syst(n,r)$. A $\mathcal{S}$-boolean partial map ($\mathcal{S}-$BPM) $A_{\mathcal{S}} : \Omega_{\mathcal{S}} \subseteq S(n,r) \to \bf{2}$ is defined as follows;\\ for $j\in \{1,\dots,p\}$,

$$
A_{\mathcal{S}}(w_j)= \left\{ \begin{array}{lll}
                  P & \textrm{if} & \sum_{i \in A_j} t_i \ge 0 \\
                  N & \textrm{if} & \sum_{i \in A_j} t_i < 0 \\
                  \end{array} \right.
$$

$A_{\mathcal{S}}(\xi_0)= A_{\mathcal{S}}(\xi_1)=\cdots=A_{\mathcal{S}}(\xi_r)=P$ and $A_{\mathcal{S}}(\eta_1)=\cdots=A_{\mathcal{S}}(\eta_{n-r})=N$.
\end{definition}

\begin{definition}
If $\mathcal{S}, \mathcal{S'} \in Syst(n,r)$, we set $\mathcal{S} \lesssim \mathcal{S'}$ if $\mathcal{S}$ is a sub-system of $\mathcal{S'}$.
\end{definition}

This obviously defines a partial order $\lesssim$ on $Syst(n,r)$.
We denote with $\mathcal{B}(n,r)$ the sub-poset of all the boolean partial maps $A\in (S(n,r) \rightsquigarrow {\bf 2})$ such that\\ $\xi_r,\cdots,\xi_1,\xi_0,\eta_1,\cdots,\eta_{n-r}\in dom(A)$ and $A(\xi_0)= A(\xi_1)=\cdots=A(\xi_r)=P$, $A(\eta_1)=\cdots=A(\eta_{n-r})=N$ and with $\mathcal{BT}(n,r)$ the subset of all the total maps of $\mathcal{B}(n,r)$. Then the map $\chi : Syst(n,r) \rightarrow \mathcal{B}(n,r)$ such that $\chi(\mathcal{S})=A_{\mathcal{S}}$, for each $\mathcal{S}\in Syst(n,r)$, is an isomorphism of posets. We denote by $\tau : \mathcal{B}(n,r) \rightarrow Syst(n,r)$ the inverse of $\chi$ and we set $\tau(A)=\mathcal{S}_A$ if $A\in \mathcal{B}(n,r)$. Obviously the restriction of $\chi$ to $\mathcal{BT}(n,r)$ defines an isomorphism between $\mathcal{BT}(n,r)$ and $TSyst(n,r)$, and we continue respectively to denote with $\chi$ and $\tau$ this isomorphism and its inverse.

Our principal question is then :

\medskip

Q1) What are the maps in $\chi(W_+CTSyst(n,r))$ and in $\chi(W_-CTSyst(n,r))$?

\medskip

Roughly speaking, what are the order-properties that characterize a boolean map $A\in \mathcal{BT}(n,r)$ in such a way that $\mathcal{S}_A$ is a compatible system?



We recall, see \cite{BisChias}, that a $(n,r)-${\it function} is an application $f:A(n,r) \to \mathbb{R}$ such that
\begin{equation}\label{(n,r)-function}
f(\tilde{r}) \ge \cdots \ge f(\tilde{1}) \ge f(0^§)=0 > f(\overline{1}) \ge \cdots \ge f(\overline{n-r}).
\end{equation}
We denote by $F(n,r)$ the set of the $(n,r)-$functions.
The function $f$ is a $(n,r)-${\it positive weight function} [{\it negative weight function}] if (\ref{(n,r)-function}) holds and if:

\begin{equation}\label{(n,r)-Wfunction}
f(\tilde{r}) + \cdots + f(\tilde{1}) + f(\overline{1}) + \cdots + f(\overline{n-r}) \ge 0 [<0].
\end{equation}
We denote by $WF_+(n,r)$ the set of the $(n,r)-$positive weight functions and with $WF_-(n,r)$ the set of the $(n,r)-$negative weight functions.

We say that a $(n,r)$-function $f$ is a {\it solution} of the system (\ref{(n,r)-system}) if the assignment

\begin{equation}\label{assignment}
x_r=f(\tilde{r}),\cdots, x_1=f(\tilde{1}), \cdots y_1=f(\overline{1}), \cdots , y_{n-r} = f(\overline{n-r})
\end{equation}

provides a solution of (\ref{(n,r)-system}).

If $f$ is a $(n,r)-$function, the {\it sum function} induced by $f$ on $S(n,r)$
$$\Sigma_f : S(n,r) \to \mathbb{R}$$
is the function that associates to $w=i_1 \cdots i_r \,\,\, |\,\,\, j_1 \cdots j_{n-r}\in S(n,r)$,
the real number $\Sigma_f (w)=f(i_1)+ \cdots + f(i_r) + f(j_1) + \cdots + f(j_{n-r}),$ see \cite{BisChias},

and we also define the map
$$
A_f : S(n,r) \to \bf{2}
$$
setting
$$
A_f (w) = \left\{ \begin{array}{lll}
                  P & \textrm{if} & \Sigma_f (w) \ge 0 \\
                  N & \textrm{if} & \Sigma_f (w) < 0.
                  \end{array} \right.
$$


If $f\in F(n,r)$, we denote by $\mathcal{S}_f$ the $(n,r)$-compatible total system having $f$ as one of its solutions and we set $Pos(f) = \{w\in S(n,r) : A_f(w)=P \}$, $Neg(f) = \{w\in S(n,r) : A_f(w)=N \}$, $\alpha_+(f) = |Pos(f)|$ and $\alpha_-(f) = |Neg(f)|$. It is obvious that $A_{\mathcal{S}_f}=A_f$.

In an attempt to answer to Q1), we give the following definition :

\begin{definition}
Let $\mathcal{H}$ be a family of maps of $\mathcal{BT}(n,r)$ and let $A\in \mathcal{H}$; we say that a boolean partial map $B\in \mathcal{B}(n,r)$ is a $\mathcal{H}$-core for $A$ if $A_{|W}=B$ (where $W=dom(B)$) and if $A'\in \mathcal{H}$ is such that $A'_{|W}=B$, then $A=A'$. We simply say that $B$ is a $\mathcal{H}$-core if it is a $\mathcal{H}$-core for some $A\in \mathcal{H}$.
\end{definition}

The following two results are very simple but they are fundamental in our strategy to approach the problem raised in Q1).


{\bf Positive local criterion (p.l.c.)} {\it Let $\mathcal{H}$ be a family of maps of $\mathcal{BT}(n,r)$ such that $\chi(W_+CTSyst(n,r)) \subseteq \mathcal{H}$ and  $\mathcal{H} \cap \chi(W_-CTSyst(n,r)) = \emptyset$, and let $A\in \mathcal{H}$.
Let $B$ denote a $\mathcal{H}$-core of $A$. Then, $\mathcal{S}_A$ is compatible if and only if $\mathcal{S}_B\in W_+CSyst(n,r)$ and, in this case, if $f\in WF_+(n,r)$ is a solution of $\mathcal{S}_B$, it is also a solution of $\mathcal{S}_A$.}
\begin{proof}
If $\mathcal{S}_A$ is compatible, then it must be necessarily $\mathcal{S}_A\in W_+CTSyst(n,r)$ because $\mathcal{H} \cap \chi(W_-CTSyst(n,r)) = \emptyset$ and $A=\chi(\mathcal{S}_A)$. Hence $\mathcal{S}_B\in W_+CSyst(n,r)$. On the other side, if $\mathcal{S}_B\in W_+CSyst(n,r)$, it has a solution $f\in WF_+(n,r)$. Then $\mathcal{S}_f \in W_+CTSyst(n,r)$ and hence, by hypothesis, $\chi(\mathcal{S}_f)\in \mathcal{H}$. It is easy to observe that $\chi(\mathcal{S}_f)=A_f$. Therefore $A_f\in \mathcal{H}$. If we denote by $W$ the domain of $B$, we have $({A_f})_{|W}= B$ since $f$ is a solution of $\mathcal{S}_B$; therefore $({A_f})_{|W}= A_{|W}$. Since $W$ is a $\mathcal{H}$-core of $A$, we have that $A=A_f$;
hence $\mathcal{S}_A$ is compatible and $f$ is one of its solution.
\end{proof}

{\bf Negative local criterion (n.l.c.)} {\it Let $\mathcal{H}$ be a family of maps of $\mathcal{BT}(n,r)$ such that $\chi(W_-CTSyst(n,r)) \subseteq \mathcal{H}$ and  $\mathcal{H} \cap \chi(W_+CTSyst(n,r)) = \emptyset$, and let $A\in \mathcal{H}$.
Let $B$ denote a $\mathcal{H}$-core of $A$. Then, $\mathcal{S}_A$ is compatible if and only if $\mathcal{S}_B\in W_-CSyst(n,r)$ and, in this case, if $f\in WF_-(n,r)$ is a solution of $\mathcal{S}_B$, it is also a solution of $\mathcal{S}_A$.}
\begin{proof}
Similar to that of p.l.c.
\end{proof}

The previous results gives us some ``local" criteria that are useful in two directions : ``from global to local" and ``from local to global".
In the direction ``from global to local", to decide if a map $A$ that we choose in a special family $\mathcal{H}$ of boolean total maps of $\mathcal{BT}(n,r)$ determines a $(n,r)$-compatible total system. In this case the previous criteria are useful if we know, for each given map $A\in \mathcal{H}$ an $\mathcal{H}$-core that is ``sufficiently" small. One of our principal results in this paper will be that of building some appropriate families $\mathcal{H}$ of boolean total maps that satisfies the previous local criteria and such that for each $A \in \mathcal{H}$ there  exists a unique $\mathcal{H}$-core of $A$ with minimal cardinality (we will call it the {\it fundamental} $\mathcal{H}$-core of $A$). We show also as this core is composed.

In the direction ``from local to global", we can ask if a given system $\mathcal{S}$ of $W_+CSyst(n,r)$ (or of $W_-CSyst(n,r)$) is equivalent to some $\mathcal{S}'\in W_+CTSyst(n,r)$ (or to some $\mathcal{S}'\in W_-CTSyst(n,r)$) and if $\mathcal{S}$ has a minimal cardinality between all the $(n,r)$-sub-systems having the same solutions of $\mathcal{S}'$.

Let us consider for example the following $(6,2)$-positively weighted system:

\begin{equation}\label{systemExample1}
  \left\{ \begin{array}{l}
                    x_2 \ge x_1 \ge 0 > y_1 \ge y_2 \ge y_3 \ge y_4 \\
                    x_1+x_2+y_1+y_2+y_3+y_4 \ge 0\\
                    x_1+y_2+y_3 < 0\\
                    \end{array}
  \right.
\end{equation}

Two different $(6,2)$-weight functions that are solutions of (\ref{systemExample1}) are the following:

$$
   \begin{array}{cccccccc}
   {} &  \tilde{2}  & \tilde{1}  & \overline{1} &  \overline{2} & \overline{3} & \overline{4} &\\
    f:  &      \downarrow & \downarrow & \downarrow &     \downarrow  & \downarrow & \downarrow     &\\
    {} &          3      &      1     &     -1     &      -1         &      -1    &      -1        &\\
  \end{array}
   \begin{array}{cccccccc}
   {} &  \tilde{2}  & \tilde{1}  & \overline{1} &  \overline{2} & \overline{3} & \overline{4} &\\
    g:  &      \downarrow & \downarrow & \downarrow &     \downarrow  & \downarrow & \downarrow     &\\
    {} &          4      &     0    &     -1     &      -1         &      -1    &      -1        &\\
  \end{array}
$$

It is easy to see that $\alpha_+(f)=40$ and $\alpha_+(g)=36$, therefore $Pos(f) \neq Pos(g)$.

On the other side, the following $(6,2)$-positively weighted system:

\begin{equation}\label{systemExample2}
  \left\{ \begin{array}{l}
                    x_2 \ge x_1 \ge 0 > y_1 \ge y_2 \ge y_3 \ge y_4 \\
                    x_1+x_2+y_1+y_2+y_3+y_4 \ge 0\\
                    x_1+y_1+y_2+y_3+y_4 \ge 0\\
                    \end{array}
  \right.
\end{equation}

has the property that for each two different $(6,2)$-weight functions $f$ and $g$ which are both solutions, it holds that $\alpha_+(f)=\alpha_+(g)=48$ and also $Pos(f)=Pos(g)$. The difference between the previous two systems is that (\ref{systemExample2}) is equivalent to a $(6,2)$-positively weighted total system, while (\ref{systemExample1}) is not.

\begin{definition}
We say that a system $\mathcal{S}\in CSyst(n,r)$ is generative if it is equivalent to some system $\mathcal{S'}\in CTSyst(n,r)$; in this case, we also say that $\mathcal{S'}$ is generated by $\mathcal{S}$, or that $\mathcal{S}$ generates $\mathcal{S'}$.
\end{definition}

Then, in the direction ``from local to global", we will see that, for a convenient family $\mathcal{H} \subseteq \mathcal{BT}(n,r)$, the boolean map determined by the system (\ref{systemExample2}) defines a $\mathcal{H}$-core and hence by p.l.c. this system is generative.

The reasons we have explained led us to search special families $\mathcal{H}$ of boolean total maps on $S(n,r)$ that satisfy the previous local criteria as well to study their $\mathcal{H}$-cores and, between them, to determine those having a minimal number of elements.
Essentially the boolean total maps that we are going to study here represent the equivalence classes of total systems: total systems can have many solutions but a unique boolean total map associated to them.\\
The proofs of our principal results does not require all the properties of the lattice $S(n,r)$, but only those of strongly involution poset. Therefore the results that we prove in the next sections hold for each finite strongly involution poset.

\section{Definitions, notations and some general results.}\label{preliminari}

Let $(X, \le)$ be a poset. If $Z \subseteq X$, we will set $\downarrow Z = \{ x \in X : \exists \,\,\, z \in Z \,\,\, \textrm{such that} \,\,\, z \ge x \} $,
$\uparrow Z = \{ x \in X : \exists \,\,\, z \in Z \,\,\, \textrm{such that} \,\,\, z \le x \}$. In particular, if $z \in X,$ we will set $\downarrow z =\downarrow \{z\} =\{ x \in X : z \ge x \}$, $\uparrow z =\uparrow \{z\} =\{ x \in X : z \le x \}$.
\begin{definition}
$J1)$ $Z$ is called a down-set of $X$ if for each $z \in Z$ and $x \in X$ with $z \ge x,$ then $x \in Z.$\\
$J2)$ $Z$ is called an up-set of $X$ if for each $z \in Z$ and $x\in X$ with $z \ge x,$ then $x \in Z.$
\end{definition}
$\downarrow Z$ is the smallest down-set of $X$ which contains $Z$ and $Z$ is a down-set in $X$ if and only if $Z=\downarrow Z.$\\
Similarly $\uparrow Z$ is the smallest up-set of $X$ which contains $Z$ and $Z$ is an up-set in $X$ if and only if $Z=\uparrow Z.$\\
Denote with $\bf{2}$ the boolean lattice composed of a chain with 2 elements that we will denote with $N$ (the minimal element) and $P$ (the maximal element). The set of all the partial maps from $X$ to $\bf{2},$ here denoted by $(X \rightsquigarrow \bf{2})$, is a poset with the following order: \\
if $(A, dom (A)), (B, dom (B)) \in (X \rightsquigarrow \bf{2}),$
\begin{equation} \label{defOrdXto2}
(A, dom (A))\trianglelefteq (B, dom (B)) \Longleftrightarrow dom (A) \subseteq dom (B) \,\,\, , B_{|dom (A)} =A.
\end{equation}

A {\it boolean partial map} (BPM) on $X$ is an element $(A, dom (A))$ of $(X \rightsquigarrow \bf{2})$, (that in the following we will denote only with $A$). If $dom (A)=X$, we will say that $A$ is a {\it boolean total map} (BTM) on $X.$

\begin{definition}
We say that a boolean partial map $A$ on $X$ is up-positive if $A^{-1}(P)$ is an up-set of $X$;
we say that it is down-negative if $A^{-1}(N)$ is a down-set of $X$.
\end{definition}
\begin{definition}
If $X$ and $Y$ are two arbitrary posets and if $A:X \to Y$ we will say that $A$ is order-preserving (OP) if for every $x_1,x_2 \in X$ such that $x_1 \le x_2$ then $A(x_1)\le A(x_2)$ in $Y.$ If $X$ is an arbitrary poset, we denote with $\mathcal{OP}(X,{\bf 2})$ the family of all op-BTM's on $X$.
\end{definition}
Given a boolean partial map $A$ on $X$, a minimal element in $A^{-1}(P)$ is called {\it minimal positive} of $A$; a maximal element in $A^{-1}(N)$
is called {\it maximal negative} of $A.$
\begin{definition}
If $Z \subset X$ we denote with $Min(Z)$ the set of minimal elements of $Z$ and with $Max(Z)$ the set of maximal elements of $Z.$
\end{definition}
If $A$ is a boolean partial map on $X$ and if $Z$ is a subset of $X,$ we set:
$$Z_{P}^A=A^{-1}(P) \cap Z = \{ x \in Z \cap dom(A) : A(x)=P \},$$
$$Z_{N}^A=A^{-1}(N) \cap Z =\{ x \in Z \cap dom(A) : A(x)=N \}.$$

With the symbol $\cup^d$ we denote the disjoint union between two sets.


The following proposition shows that the concepts of up-positivity, down-negativity and of order-preserving are equivalent for boolean total maps.

\begin{proposition}\label{equivopup}
Let $X$ be an arbitrary poset and $A$ a $BTM$ on $X.$ Then the following are equivalent:

i) $A$ is order-preserving (op);

ii) $A$ is up-positive (up);

iii) $A$ is down-negative (dn).

\end{proposition}
\begin{proof}
We prove that $i)$ and $ii)$ are equivalent. The equivalence of $i)$ and $iii)$ follows likewise.\\
$i) \Rightarrow ii):$ suppose that $x_1,x_2 \in X$ and $A(x_1)=P.$ Since $A$ is order-preserving, we have that $A(x_1) \le A(x_2)$ and in ${\bf 2}$ this implies that $A(x_2)=P.$ \\
$ii) \Rightarrow i):$ let $x_1,x_2 \in X$ such that $x_1 \le x_2$ and suppose, on the contrary, that $A(x_1) \not\le A(x_2).$ Since ${\bf 2}$ is totally ordered, we have that $A(x_1) >A(x_2);$ hence $A(x_1)=P$ and $A(x_2)=N.$ Since $x_1 \le x_2$ and $A(x_1)=P,$ and since $A$ is up-positive it follows that $A(x_2)=P,$ which is a contradiction because $A(x_2)=N.$
\end{proof}

Our next question is the following : let $X$ be a poset and let $W \subset X.$ Given a function $\varphi: W \to {\bf 2},$ which properties the couple $(W,\varphi)$ has to have in order that there exists a unique particular type of BTM $A$ on $X$ which extends $\varphi$ ?

\begin{definition}
Let $X$ be an arbitrary poset and let $\mathcal{H}$ a family of BTM's on $X$. A couple $(W,\varphi)$ is a $\mathcal{H}$-core on $X$ if:
\begin{itemize}
\item[N1)] $\varphi$ is a BPM on $X$ such that $dom(\varphi)=W;$\\
\item[N2)] there exists a unique $A \in \mathcal{H}$ such that $A_{|W}=\varphi.$
\end{itemize}
\end{definition}

If $(W, \varphi)$ is a $\mathcal{H}$-core on $X$, the unique map $A \in \mathcal{H}$ (in N2)) which extends $\varphi$ is called the $\mathcal{H}$- {\it map spanned} by the core $(W, \varphi)$. We also say that $(W, \varphi)$ {\it spans} $A$ and sometime we write $A = A_{W,\varphi}$ to mean that $A$ is spanned by $(W, \varphi)$. On the other side, if $A \in \mathcal{H}$ is given, we say that a subset $W$ of $X$ is a $\mathcal{H}$-{\it core for} $A$ if the couple $(W, \varphi)$, with $\varphi = A_{|W}$, is a $\mathcal{H}$-core on $X$ (in this case $A$ is obviously the unique map in $\mathcal{H}$ spanned by $(W, \varphi)$).
\begin{definition}
We say that $W$ is a $\mathcal{H}$-fundamental core for $A$ if it is a $\mathcal{H}$-core for $A$ and if, for each $\mathcal{H}$-core $V$ for $A,$ $W \subseteq V$.
\end{definition}
Obviously, if there exists a $\mathcal{H}$-fundamental core of $A$, then it is unique, therefore we can speak of {\it the} $\mathcal{H}$-fundamental core for $A$. If the family $\mathcal{H}$ is clear from the context, we say simply core instead of $\mathcal{H}$-core.

Let $Core_\mathcal{H}(X)$ be the family of all the $\mathcal{H}$-cores $(W,\varphi)$ on $X$. We consider the function $f:Core_\mathcal{H}(X) \to \mathcal{H}$ such that
$$f((W, \varphi))=A_{W,\varphi}.$$
Naturally $f$ is surjective, since for all $A \in \mathcal{H}$ we have $A=A_{X,A}=f((X,A))$ and the core for $A$ is the obvious one $(W,\varphi)=(X,A).$ \\
We define on $Core_\mathcal{H}(X)$ the following relation
$$(W_1,\varphi_1) \sim (W_2, \varphi_2) \Leftrightarrow A_{W_1,\varphi_1}=A_{W_2,\varphi_2}.$$
Then $\sim$ is an equivalence relation and by the universal property of the quotient there exists a unique injective map $f^*$ from $Core_\mathcal{H}(X)/\sim$ into $\mathcal{H}$ induced by $f$. Since $f$ is surjective, it follows that $f^*$ is bijective.

In this paper we study the cores for two particular families of BTM's and we determine explicitly the set $Core_\mathcal{H}(X)/\sim$ for these families of boolean maps.
The two families of BTM's that we will examine are defined on a particular class of posets, which are the involution posets.

\begin{definition}
An involution poset (IP) is a poset $(X, \le)$  with a unary mapping $c:X \to X$ such that:
\begin{itemize}
\item[$I1)$] $c(c(x))=x,$ for all $x \in X$; \
\item[$I2)$] if $x,y \in X$ and if $x \le y,$ then $c(y)\le c(x)$.\
\end{itemize}
\end{definition}
If $x \in X,$ we will write $c(x)=x^c.$ \\

In this paper we consider involution posets $(X, \le, c)$ having the following further property :
\begin{itemize}
\item[$I3)$] $x^c \neq x$ for all $x\in X$ (if $|X| \ge 2$).\
\end{itemize}
\begin{definition}
We call strongly involution poset (SIP) an involution poset \\ $(X, \le, c)$ which satisfies $I3)$.
\end{definition}
Let us observe that if $X$ is an involution poset, by $I1)$ follows that $c$ is bijective and by $I1)$ and $I2)$ it holds that if $x,y \in X$ are such that $x<y,$ then $y^c < x^c.$

\medskip
If $(X, \le, c)$ is an involution poset and if $Z\subseteq X$, we will set $Z^c=\{ z^c : z \in Z \}.$ \\

\begin{definition}
If X is a SIP, will say that a boolean partial map $A$ on $X$ is :
\begin{itemize}
\item[i)] complemented positive if $A^{-1}(N)^c \subseteq A^{-1}(P)$; \\
\item[ii)] complemented negative if $A^{-1}(P)^c \subseteq A^{-1}(N)$.
\end{itemize}
\end{definition}

\begin{definition}
If $X$ is a SIP, a BPM $A$ on $X$ is called positively weighted boolean partial map (briefly +WBPM) if it is up-positive, down-negative and complemented-positive; in particular, if $A$ is also total on $X$, it is called positively weighted boolean total map (briefly +WBTM).
\end{definition}
Similar definitions holds when complemented-positive is replaced with complemented-negative, +WBPM with -WBPM and +WBTM with -WBPM. A WBPM is a +WBPM or a -WBPM, a WBTM is a +WBTM or a -WBTM.

If $X$ is a SIP, we denote by $\mathcal{W}_+(X,{\bf 2})$ the family of all the +WBTM's on $X$ and by $\mathcal{W}_-(X,{\bf 2})$ that of all the -WBTM's on $X$. Then $\mathcal{W}_+(X,{\bf 2})$ and $\mathcal{W}_-(X,{\bf 2})$ are the two families that we will study in this paper. Obviously, if $X$ is a SIP, by virtue of Proposition \ref{equivopup}, it follows that $\mathcal{W}_+(X,{\bf 2})$ [$\mathcal{W}_-(X,{\bf 2})$] is the sub-family of all the maps in $\mathcal{OP}(X,{\bf 2})$ which are also complemented positive [negative].

By Proposition \ref{equivopup} it follows that if $A$ is a BTM on $X$, then $A$ is a +WBTM [-WBTM] if and only if $A$ is up-positive and complemented positive [negative].
The following proposition shows that each boolean lattice is also a SIP and that a boolean lattices morphism is a +WBTM and a -WBTM.

\begin{proposition}
Let $(X,\wedge,\vee,0,1,')$ be a boolean lattice, then $X$ is a SIP. Moreover, if  $A:X\to {\bf 2}$ is a boolean lattices morphism, then $A$ is a +WBTM and a -WBTM.
\end{proposition}
\begin{proof}
Let $c:X \to X$ be such that $c(x)=x'$ where $x'$ is the complement of $x$ in $X,$ i.e. the unique element of $X$ such that $x \wedge x'=0$ and $x \vee x'=1.$ By the well-known properties of the function $x \mapsto x',$ it follows that $c$ satisfies the properties $I1)\,\,I2)\,\,I3)$.


By definition of morphism of boolean lattices, $A$ is such that
$$A(a\vee b)=A(a)\vee A(b), \,\,\,\,\,\, A(a\wedge b)=A(a)\wedge A(b),$$
$$A(0)=0, \,\,\,\,\, A(1)=1, \,\,\,\,\,\, A(a')=(A(a))'.$$
It is well-known by the general theory that $A$ is order-preserving (hence also up-positive and down-negative). Finally, if $x \in X$ is such that
$A(x)=N,$ then $A(x^c)=A(x')=A(x)'=N'=P$ because in ${\bf 2}$ the complement of $N$ is $P.$ Hence $A$ is complemented positive. Similarly we see that $A$ is also complemented negative.
\end{proof}

\begin{proposition}\label{minimax}
Let $X$ be a SIP and $A$ a +WBPM on $X$, then :

i) if $w$ is a minimal positive of $A$ such that $A(w^c)=N$, it follows that $w^c$ is a maximal negative of $A$;

ii) if $x,\,\,\, x^c \in dom(A)$ and $x^c \le x,$ then $A(x)=P$.

If $A$ is a -WBPM on $X$, then :

i') if $w$ is a maximal negative of $A$ such that $A(w^c)=P$, it follows that $w^c$ is a minimal positive of $A$;

ii') if $x,\, x^c \in dom(A)$ and $x^c \le x,$ then $A(x^c)=N$.
\end{proposition}
\begin{proof}
$i)$ Suppose by contradiction that $w^c$ isn't a maximal negative of $A.$ Then there exists an element $w' \in A^{-1} (N)$ such that $w' >w^c.$ Since $A$ is complemented positive, we have that $A^{-1}(N)^c \subseteq A^{-1} (P)$ and hence $(w')^c \in A^{-1}(P).$ Furthermore, since $w' >w^c,$ we also have that $w=(w^c)^c > (w')^c$ and this is a contradiction by the minimality of $w$ in $A^{-1}(P).$

$ii)$ Suppose by contradiction that $A(x)=N.$ Since $A$ is complemented positive, we have that $x^c \in dom(A)$ and $A(x^c)=P.$ Since $x^c \le x$ and $A$ is up-positive, we have that $A(x)=P$ and this is a contradiction.

The proof of $i')$ and $ii')$ is similar.
\end{proof}

By Proposition \ref{minimax}-ii) and Proposition \ref{minimax}-ii'),
\begin{definition}
The elements $w \in X$ such that $w^c \sqsubseteq w$ are called \textit{complemented}.
\end{definition}

\section{The $\mathcal{H}$-cores when  $\mathcal{H}=\mathcal{W}_+(X,{\bf 2})$ or $\mathcal{H}=\mathcal{W}_-(X,{\bf 2})$}
In this section we assume that $X$ is a finite strongly involution poset and we determine the  $\mathcal{W}_+(X,{\bf 2})$- fundamental core of an arbitrary +WBTM of $\mathcal{W}_+(X,{\bf 2})$ and the $\mathcal{W}_-(X,{\bf 2})$- fundamental core of an arbitrary -WBTM of $\mathcal{W}_-(X,{\bf 2})$.

\begin{proposition}\label{negincl}
i) Let $A\in \mathcal{W}_+(X,{\bf 2})$ and $W$ a core for $A,$ then
$$X^A_N \subseteq \downarrow W_N^A.$$

ii) Let $A\in \mathcal{W}_-(X,{\bf 2})$ and $W$ a core for $A,$ then
$$X^A_P \subseteq \uparrow W_P^A.$$
\end{proposition}
\begin{proof}
$i)$ Let $w \in X^A_N$ be fixed. In order to prove that $w \in \downarrow W^A_N$ we need to exhibit an element $\overline{w} \in W^A_N$ such that $\overline{w} \ge w.$ If $w \in W,$ then $w\in X^A_N \cap W=W^A_N,$ and in this case the assertion follows taking $\overline{w}=w.$ \\
Suppose that $w \not\in W$ and that, by contradiction, $w \not\in \downarrow W^A_N.$ This implies that the set $\mathcal{F}=\{ u \in X^A_N: u \not\in W, u\not\in \downarrow W^A_N \}=X^A_N \setminus \downarrow W^A_N$ is not empty, since $w \in \mathcal{F}.$ Let $z$ be a maximal element of $\mathcal{F}.$ We define a function $A':X \to {\bf 2}$ such that
$$
A'(u)= \left\{ \begin{array}{lll}
               A(u) & \textrm{if} & u \neq z \\
               P  & \textrm{if} & u=z
               \end{array} \right.
$$
We prove that $A'$ is a +WBTM. \\
{\it Step 1} We prove that $A'$ is up-positive, i.e. that $X^{A'}_P$ is an up-set in $X.$
Let $w_1,w_2 \in X$ be such that $w_1 \le w_2,$ and suppose that $A'(w_1)=P.$ We want to prove that $A'(w_2)=P.$
\begin{itemize}
\item[$i_1)$] Suppose that $w_1 \neq z$ and $w_2 \neq z.$ Then, since $A'(w_1)=P$ and $w_1 \neq z,$ by the definition of $A',$ it follows that $A(w_1)=P.$ Since $A$ is a +WBTM and $w_1 \le w_2$ it follows that $A(w_2)=P$ and since $w_2 \neq z,$ always by the definition of $A',$ it follows that $A'(w_2)=A(w_2)=P.$\\
\item[$i_2)$] Suppose that $w_1 < w_2=z.$ In this case we have that $A'(w_2)=A'(z)=P$ by definition of $A',$ and hence the assertion follows.\\
\item[$i_3)$] Suppose that $w_1=z <w_2.$ Suppose, by contradiction, that $A'(w_2)=N.$ Since $w_2 \neq z,$ by definition of $A'$ we will have that $A'(w_2)=A(w_2)=N.$ Observe that $w_2 \not\in \downarrow W^A_N.$ Indeed, if $w_2 \in \downarrow W^A_N,$ then there exists an element $\overline{w} \in W^A_N$ such that $w_2 \le \overline{w}$ and hence we will have that $z<w_2 \le \overline{w},$ from which it follows that $z \in \downarrow W^A_N,$ against the hypothesis that $z \in \mathcal{F}.$ Hence we have that $A(w_2)=N$ and $w_2 \not\in \downarrow W^A_N$ i.e. $w_2 \in \mathcal{F},$ but this is in contradiction with the maximality of $z$ in $\mathcal{F}.$
\end{itemize}
Hence $A'$ is up-positive.\\
{\it Step 2} We prove that $A'$ is complemented positive, i.e. that $(X^{A'}_N)^c \subseteq X^{A'}_P.$ Let $w \in X$ be fixed and such that $A'(w)=N.$ We prove that $A'(w^c)=P.$
\begin{itemize}
\item[$j_1)$] Suppose that $w \neq z$ and $w \neq z^c.$ In this case, since $w\neq z,$ by definition of $A',$ it follows that $A(w)=A'(w)=N.$
Since $A$ is complemented positive, we have that $A(w^c)=P.$ On the other hand, since $w \neq z^c,$ we have that $w^c \neq z$ and hence, by definition of $A',$ we have that $A'(w^c)=A(w^c)=P.$ \\
\item[$j_2)$] Suppose that $w=z$ and hence that $w^c=z^c.$ In this case, since $A'(w)=A'(z)=P,$ the hypothesis $A'(w)=N$ is empty and hence there is nothing to prove. \\
\item[$j_3)$] Suppose $w \neq z$ and $w=z^c.$ In this case we have that $z^c=w\neq z$ and hence, by definition of $A',$ it follows that $A'(z^c)=A(z^c).$ Since $z \in \mathcal{F},$ we have that $A(z)=N$ and hence, since $A$ is complemented positive, we will have that $A(z^c)=P$ and that $A'(w)=A'(z^c)=A(z^c)=P.$ This means that the hypothesis $A'(w)=N$ is empty, if $w \neq z$ and $w=z^c.$ This prove that $A'$ is complemented positive.
\end{itemize}
Hence $A'$ is a +WBTM on $X$ which coincides with $A$ on $X$ everywhere except on $z$ where we have $A'(z)=P$ and $A(z)=N.$ Since $z \not\in W$ (because $z \in \mathcal{F}$) it holds that $A'_{|W}=A_{|W}$ but $A\neq A'$ on $X.$ This contradicts the hypothesis that $W$ is a core for $A.$

$ii)$ Similar arguments apply.
\end{proof}
\begin{corollary}\label{negug}
i) If $A\in \mathcal{W}_+(X,{\bf 2})$ and $W$ is a core for $A,$ then
$$X^A_N=\downarrow W^A_N.$$

ii) If $A\in \mathcal{W}_-(X,{\bf 2})$ and $W$ is a core for $A,$ then
$$X^A_P=\uparrow W^A_P.$$

\end{corollary}
\begin{proof}
$i)$ Let $\overline{w} \in \downarrow W^A_N.$ By definition of down-set, there exists an element $w\in W^A_N$ such that $\overline{w} \le w.$ Since $w\in W^A_N=A^{-1}(N)\cap W,$ we have that $A(w)=N,$ and since $A$ is down-negative, it follows that $A(\overline{w})=N,$ i.e. $\overline{w} \in X^A_N.$ Therefore $\downarrow W^A_N \subseteq X^A_N.$ Then the assertion follows by Proposition \ref{negincl}-$i)$.

$ii)$ This follows by the same reasoning of $i)$ and using then Proposition \ref{negincl}-$ii)$.
\end{proof}
\begin{corollary}\label{maxneg}
$i)$ If $A\in \mathcal{W}_+(X,{\bf 2})$  and $W$ is a core for $A,$ then $Max(X^A_N) \subseteq W_N^A.$

$ii)$ If $A\in \mathcal{W}_-(X,{\bf 2})$  and $W$ is a core for $A,$ then $Min(X^A_P) \subseteq W_P^A.$
\end{corollary}
\begin{proof}
$i)$ Let $w\in Max(X^A_N),$ by Proposition \ref{negincl}-$i)$ it follows that $w \in \downarrow W^A_N.$ This implies that there exists an element $\overline{w} \in W^A_N$ such that $\overline{w} \ge w.$ By the maximality of $w$ in $X^A_N,$ it holds that $w=\overline{w} \in W^A_N.$

$ii)$ Analogously using Proposition \ref{negincl}-$ii)$.
\end{proof}
\begin{corollary}\label{maxneg2}
i) Let $A\in \mathcal{W}_+(X,{\bf 2})$. Let $W$ be a core for $A$ on $X$ such that $W^A_N$ is an anti-chain on $X,$ then
$$Max(X^A_N)=W^A_N.$$

ii) Let $A\in \mathcal{W}_-(X,{\bf 2})$. Let $W$ be a core for $A$ on $X$ such that $W^A_P$ is an anti-chain on $X,$ then
$$Min(X^A_P)=W_P^A.$$
\end{corollary}
\begin{proof}
$i)$ By Corollary \ref{maxneg}-$i)$, we need to prove the inclusion $Max(X^A_N) \supseteq W^A_N.$ Suppose that $w \in W^A_N$ and that (by contradiction) $w \not\in Max(X^A_N).$ Then there exists an element $\overline{w} \in X^A_N$ such that $\overline{w} > w.$ By Corollary \ref{negug}-$i)$, we have that $\overline{w} \in \downarrow W^A_N.$ Then there exists an element $\tilde{w} \in W^A_N$ such that $\tilde{w} \ge \overline{w},$ and hence we have that
$\tilde{w} \ge \overline{w} > w,$ with $\tilde{w},\,\,\, w \in W^A_N,$ and this contradicts the hypothesis that $W^A_N$ is an anti-chain on $X.$

$ii)$ Similar analysis using Corollary \ref{maxneg}-$ii)$ and Corollary \ref{negug}-$ii)$.
\end{proof}


\begin{proposition}\label{posincl}
i) Let $A\in \mathcal{W}_+(X,{\bf 2})$ and $W$ a core for $A.$ Then
$$X^A_P \subseteq (\uparrow W^A_P) \cup (\uparrow(W_N^A)^c).$$

ii) Let $A\in \mathcal{W}_-(X,{\bf 2})$ and $W$ a core for $A.$ Then
$$X^A_N \subseteq (\downarrow W^A_N) \cup (\downarrow(W_P^A)^c).$$
\end{proposition}
\begin{proof}
$i)$ First of all let $w \in W\cap X^A_P$ then $w\in W^A_P \subseteq \uparrow W^A_P.$ Then assume that $w \not\in W.$ Suppose by contradiction that $w\not\in \uparrow W^A_P$ and $w\not\in \uparrow ((W^A_N)^c).$ Set
$$\mathcal{F}=\{ u \in X^A_P : u\not\in W, u \not\in \uparrow W^A_P, u \not\in \uparrow ((W^A_N)^c) \}=X^A_P \setminus (\uparrow W^A_P \cup \uparrow (W^A_N)^c)).$$
It follows that $w\in \mathcal{F}$ and that $\mathcal{F}$ is not empty. \\Let $z$ be a minimal element of $\mathcal{F}.$ Define $A':X \to {\bf 2}$
setting:
$$A'(u)=\left\{ \begin{array}{lll}
                A(u) & \textrm{if} & u \neq z \\
                N    & \textrm{if} & u=z.
                \end{array} \right.
$$
{\it Step 1}
We need to prove that $A'$ is a +WBTM on $X.$ This will conduce to a contradiction because $A_{|W}=A'_{|W}$ (note that $z\not\in W$) and $A'(z)=N\neq P =A(z)$ and because by hypothesis $W$ is a core for $A.$
We start proving that $A'$ is complemented positive. Suppose that $v \in X$ and that $A'(v)=N.$ We need to prove that $A'(v^c)=P.$
\begin{itemize}
\item[1)] Suppose that $v\neq z$ and $v \neq z^c.$ In this case, we have that $N=A'(v)=A(v).$ Since $A$ is complemented positive, we have that $A(v^c)=P,$ and hence, since $v^c \neq z,$ by definition of $A'$ we have that $A'(v^c)=A(v^c)=P.$ \\
    Before analyzing the remaining cases, we prove that $A'(z^c)=P.$ Indeed, by contradiction, suppose that $A'(z^c)=N;$ since $z^c \neq z$ we will have that $A(z^c)=A'(z^c)=N,$ and hence $z^c \in X^A_N.$ By Corollary \ref{negug}-i), it follows that $z^c \in \downarrow W^A_N.$ Hence there exists an element $\overline{z} \in W^A_N$ such that $z^c \le \overline{z},$ and hence: $\overline{z}^c \le (z^c)^c=z,$ with $\overline{z}^c \in (W^A_N)^c.$ Therefore $z \in \uparrow (W_N^A)^c,$ against the hypothesis that $z \in\mathcal{F}.$ \\
\item[2)] Suppose that $v=z.$ In this case $A'(v^c)=A'(z^c)=P.$\\
\item[3)] Suppose that $v=z^c.$ In this case the hypothesis $A'(v)=N$ is empty because we have proved that $A'(z^c)=P.$
\end{itemize}
Hence $A'$ is complemented positive.
\\{\it Step 2} We prove now that $A'$ is up-positive. Let $w_1,w_2 \in X$ be such that $w_1 \le w_2$ and suppose that $A'(w_1)=P.$ We need to prove that $A'(w_2)=P.$
\begin{itemize}
\item[1)] Suppose that $w_1 \neq z$ and $w_2 \neq z.$ In this case $A(w_1)=A'(w_1)=P,$ and since $w_1 \le w_2$ and $A$ is up-positive, we will have that $A(w_2)=P.$ Since $w_2 \neq z,$ it follows that $A'(w_2)=A(w_2)=P.$ \\
\item[2)] Suppose that $w_1=z < w_2.$ In this case we will have that $A'(z)=N,$ and hence the hypothesis $A'(w_1)=P$ is empty.\\
\item[3)] Suppose that $w_1 < w_2=z.$ In this case we have that $A'(w_2)=A'(z)=N.$ We will show that the hypothesis $A'(w_1)=P$ leads to a contradiction. Suppose that $A'(w_1)=P.$ Since $w_1 \neq z,$ we will have that $A(w_1)=A'(w_1)=P,$ i.e. that $w_1 \in X^A_P.$ On the other hand, $w_1 \not\in \uparrow W^A_P \cup \uparrow (W^A_N)^c.$ Indeed if $w_1 \in \uparrow W^A_P \cup \uparrow (W^A_N)^c,$ since $\uparrow W^A_P$ and $\uparrow(W^A_N)^c$ are up-sets, from $w_1 < z$ it holds that $z \in \uparrow W^A_P \cup \uparrow (W^A_N)^c,$ in contradiction with $z \in \mathcal{F}.$ Therefore $w_1 \in \mathcal{F}$ and $w_1 < z$ and this is in contradiction with the minimality of $z \in \mathcal{F}.$
\end{itemize}
Hence $A'$ is up-positive.

In the same manner we can prove $ii)$, using Corollary \ref{negug}-$ii)$.
\end{proof}
\begin{corollary} \label{posug}
i) Let $A\in \mathcal{W}_+(X,{\bf 2})$ and $W$ a core for $A.$ Then
$$X^A_P=(\uparrow W_P^A) \cup (\uparrow(W_N^A)^c).$$

ii) Let $A\in \mathcal{W}_-(X,{\bf 2})$ and $W$ a core for $A.$ Then
$$X^A_N=(\downarrow W^A_N) \cup (\downarrow(W_P^A)^c).$$
\end{corollary}
\begin{proof}
$i)$ Let $\overline{w} \in \uparrow W^A_P \cup \uparrow (W^A_N)^c.$ Then there exists an element $w \in W^A_P \cup (W^A_N)^c$ such that $w \le \overline{w}.$ Since $A(w)=P$ and $A$ is up-positive, we have that $A(\overline{w})=P,$ i.e. $\overline{w} \in X^A_P.$ This implies that $\uparrow W^A_P \cup \uparrow (W^A_N)^c \subseteq X^A_P.$ The assertion follows by Proposition \ref{posincl}-$i)$.

$ii)$ Similarly using Proposition \ref{posincl}-$ii)$
\end{proof}
\begin{theorem} \label{generators}
i) Let $A\in \mathcal{W}_+(X,{\bf 2})$  and $W$ a core for $A,$ then
$$X=[ \uparrow W_P^A \cup \uparrow((W_N^A)^c) ]\cup^d \downarrow W^A_N .$$

ii) Let $A\in \mathcal{W}_-(X,{\bf 2})$  and $W$ a core for $A,$ then
$$X=[ (\downarrow W^A_N) \cup (\downarrow(W_P^A)^c) ]\cup^d \uparrow W^A_P .$$
\end{theorem}
\begin{proof}
i) Since $X=X^A_P \cup^d X^A_N,$ the assertion is a direct consequence of Corollaries \ref{negug}-$i)$ and \ref{posug}-$i)$.

ii) Similarly using the Corollaries \ref{negug}-$ii)$ and \ref{posug}-$ii)$.
\end{proof}
\begin{proposition} \label{minincl}
i) Let $A\in \mathcal{W}_+(X,{\bf 2})$ and $W$ a core for $A$, then
$$Min(X^A_P) \subseteq W^A_P \cup (W^A_N)^c.$$

ii) Let $A\in \mathcal{W}_-(X,{\bf 2})$ and $W$ a core for $A$, then
$$Max(X^A_N) \subseteq W^A_N \cup (W^A_P)^c.$$
\end{proposition}
\begin{proof}
$i)$ Suppose that $w \in Min(X^A_P).$ By Corollary \ref{posug}-$i)$, it follows that $w \in \uparrow W^A_P$ or $w \in \uparrow (W^A_N)^c.$ If $w \in \uparrow W^A_P,$ then there exists $\overline{w} \in W^A_P$ such that $\overline{w} \le w.$ By the minimality of $w$ in $X^A_P,$ we will have that $w=\overline{w} \in W^A_P.$ If $w \in \uparrow (W^A_N)^c,$ then there exists $\overline{w} \in (W^A_N)^c$ such that $\overline{w} \le w.$ By the minimality of $w$ in $X^A_P,$ it follows also that $w=\overline{w} \in (W^A_N)^c.$ Hence $Min(X^A_P) \subseteq W^A_P \cup (W^A_N)^c.$

$ii)$ Likewise using Corollary \ref{posug}-$ii)$.
\end{proof}
\begin{proposition}\label{minmin}
i) Let $A\in \mathcal{W}_+(X,{\bf 2})$ and $W$ a core for $A$ ,
then
$$Min(X^A_P)=Min(W^A_P \cup (W^A_N)^c).$$

ii) Let $A\in \mathcal{W}_-(X,{\bf 2})$ and $W$ a core for $A$ ,
then
$$Max(X^A_N)=Max(W^A_N \cup (W^A_P)^c).$$
\end{proposition}
\begin{proof}
$i)$ We start proving the inclusion $Min(X^A_P) \subseteq Min(W^A_P \cup (W^A_N)^c).$ Suppose that $z \in Min(X^A_P).$ By Proposition \ref{minincl}-$i)$, it follows that $z \in W^A_P \cup (W^A_N)^c.$ Suppose by contradiction that $z\not\in Min (W^A_P \cup (W^A_N)^c),$ then there exists $\overline{z} \in W^A_P \cup (W^A_N)^c$ such that $\overline{z} <z.$ Since $A(\overline{z})=P,$ this contradicts the minimality of $z$ in $X^A_P.$ Hence $Min(X^A_P) \subseteq Min(W^A_P \cup (W^A_N)^c).$ \\
Now we prove the other inclusion $Min(X^A_P) \supseteq Min(W^A_P \cup (W^A_N)^c).$ Suppose that $w \in Min(W^A_P \cup (W^A_N)^c).$ Obviously $w\in X^A_P.$ If, by contradiction, $w\not\in Min (X^A_P),$ then there exists an element $\overline{w} \in X^A_P$ such that $\overline{w} <w.$ By Corollary \ref{posug}-$i)$ it follows that either $\overline{w} \in \uparrow W^A_P$ or $\overline{w} \in \uparrow (W^A_N)^c.$
\begin{itemize}
\item[1)] if $\overline{w} \in \uparrow W^A_P,$ there exists an element $\tilde{w} \in W^A_P$ such that $\tilde{w} \le \overline{w},$
therefore $\tilde{w} \le \overline{w} <w,$ and this contradicts that $w \in Min(W^A_P \cup (W^A_N)^c).$ \\
\item[2)] if $\overline{w} \in \uparrow (W^A_N)^c,$ then there exists an element $\tilde{w} \in (W^A_N)^c$ such that $\tilde{w} \le \overline{w},$ hence $\tilde{w} \le \overline{w} < w,$ and this is a contradiction to $w\in Min(W^A_P \cup (W^A_N)^c).$
\end{itemize}

$ii)$ Analogously using Proposition \ref{minincl}-$ii)$ and Corollary \ref{posug}-$ii)$.
\end{proof}
\begin{proposition}\label{primoNucleo}
Let  $A$ be a +WBTM or a -WBTM on $X.$ Then, setting
$$N(A)=Min(X^A_P)\cup Max(X^A_N),$$
it follows that $N(A)$ is a core for $A$ on $X.$
\end{proposition}
\begin{proof}
Let $Min(A^{-1}(P))=\{ w_1,w_2, \cdots, w_k \}$ and $Max(A^{-1}(N))=\\ \{ v_1, \cdots, v_q \},$ hence $N(A)=\{ w_1, \cdots w_k,v_1, \cdots, v_q\}.$ We start observing that $A^{-1}(P)$ and $A^{-1}(N)$ are two anti-chains in $X,$ because they are, respectively, the minimal elements of $A^{-1}(P)$ and the maximal elements of $A^{-1}(N).$ Let $A'$ be an other +WBTM on $X$ such that $A'_{|N(A)}=A_{|N(A)}$ i.e. such that $A'(w_1)=\cdots=A'(w_k)=P$
and $A'(v_1)= \cdots = A'(v_q)=N.$ Suppose, by contradiction, that $A' \neq A.$ Then there exists an element $w\in X$ such that $A(w) \neq A(w').$
Then we have two possibilities:
\begin{itemize}
\item[$i)$] $A(w)=P$ and $A'(w)=N.$ In this case, $w \in A^{-1}(P)$ and hence, since $A^{-1}(P)=\cup_{i=1}^{k} (\uparrow w_i),$ there exists $i \in \{ 1, \cdots, k \}$ such that $w\in (\uparrow w_i),$ i.e. such that $w_i \le w.$ Since $A'(w_i)=P,$ we will have that $w_i \in (A')^{-1}(P)$ and since $A'$ is up-positive and $w_i \le w,$ it follows that $w \in (A')^{-1}(P)$ i.e. $A'(w)=P,$ and this is a contradiction. \\
\item[$ii)$] $A(w)=N$ and $A'(w)=P.$ In this case, $w \in A^{-1}(N),$ and hence, by $A^{-1}(N)=\cup_{j=1}^q (\downarrow v_j),$ there exists $j \in \{ 1, \cdots, q\}$ such that $w \in (\downarrow v_j),$ i.e. such that $w \le v_j.$ Since $v_j \in (A')^{-1}(N)$ and since $A'$ is down-negative, it follows that $w \in (A')^{-1}(N)$ or equivalently $A'(w)=N$ which is a contradiction.
\end{itemize}
Therefore $A=A',$ and hence $N(A)$ is a core for $A.$
\end{proof}

The next result shows that each +WBTM and each -WBTM have a unique minimal core and it also describes such a core.

\begin{theorem}\label{NA}
i) Let $A\in \mathcal{W}_+(X,{\bf 2})$. Then, setting
$$Core_{w+}(A)=N(A)\setminus [Max(A^{-1}(N))]^c,$$
it results that $Core_{w+}(A)$ is the $\mathcal{W}_+(X,{\bf 2})$-fundamental core for $A$.

ii) Let $A\in \mathcal{W}_-(X,{\bf 2})$. Then, setting
$$Core_{w-}(A)=N(A)\setminus [Min(A^{-1}(P))]^c,$$
it results that $Core_{w-}(A)$ is the $\mathcal{W}_-(X,{\bf 2})$-fundamental core for $A$.
\end{theorem}
\begin{proof}
$i)$ Let $Min(A^{-1}(P))=\{ w_1, \cdots , w_k \}$ and $Max(A^{-1}(N))=\{ v_1, \cdots , v_q \};$ then $\{ v_1^c, \cdots, v_q^c \} \subseteq A^{-1}(P).$
If $\{ v_1^c, \cdots, v_q^c \} \cap \{ w_1, \cdots , w_k \}= \emptyset$ then $Core_{w+}(A)$ coincide with $N(A)$ and hence the assertion follows by Proposition \ref{primoNucleo}.

If $\{ v_1^c, \cdots, v_q^c \} \cap \{ w_1, \cdots, w_k \} \neq \emptyset$ we assume, without loss of generality, that

$\{ v_1^c, \cdots , v_q^c \} \cap \{w_1, \cdots , w_k \}=\{ w_1, \cdots , w_p \}$, for some $p$ such that $1 \le p \le \min(k,q).$

Re-ordering the indexes, we can assume that
$w_1=v_1^c, \cdots ,w_p=v_p^c.$ Then we have that:
$$Core_{w+}(A)=\{ w_{p+1}, \cdots , w_k, v_1, \cdots, v_p,v_{p+1}, \cdots , v_q \}.$$
Observe that $Core_{w+}(A) \cap A^{-1}(P) = \{ w_{p+1}, \cdots , w_k \}$ is an anti-chain in $X,$ because it is a subset of the anti-chain $Min(A^{-1}(P));$ $Core_{w+}(A) \cap A^{-1}(N)= \{ v_1 , \cdots, v_q \}$ is an anti-chain in $X$ because it coincides with the anti-chain $Max(A^{-1}(N)).$ Let $A'$ be an other +WBTM on $X$ and suppose that $A'(w_{p+1})=\cdots =A'(w_k)=P,$ and $A'(v_1)=\cdots =A'(v_q)=N.$
We need to prove that $A'=A$ on all $X.$ Suppose by contradiction that $A \neq A'$ on $X,$ then there exists an element $w \in X$ such that $A(w) \neq A'(w).$ First suppose that $A'(w)=P$ and $A(w)=N.$ In this case, $w \in A^{-1}(N)$ and hence by $A^{-1}(N)=\cup_{j=1}^q (\downarrow v_j)$ there exists $j \in \{ 1, \cdots q \}$ such that $w \in \downarrow v_j$ i.e. such that $w \le v_j.$ Since $v_j \in (A')^{-1}(N)$ and since $A'$ is down-negative, by $w \le v_j$ it follows that $w \in (A')^{-1}(N)$ i.e. $A'(w)=N,$ which is a contradiction.\\
Finally suppose that $A'(w)=N$ and $A(w)=P.$ In this case, since $w \in A^{-1}(P)$ and since $A^{-1}(P)=\cup_{i=1}^k (\uparrow w_i)$ there exists $i \in \{ 1, \cdots, k \}$ such that $w \in \uparrow w_i,$ i.e. such that $w_i \le w.$ We distinguish two cases:
\begin{itemize}
\item[$j_1)$] if $i \in \{ p+1, \cdots, k \},$ then $w_i \in (A')^{-1}(P)$ and since $A'$ is up-positive, by $w_i \le w$ it follows that $w \in (A')^{-1}(P)$ i.e. $A'(w)=P,$ and this is a contradiction. \\
\item[$j_2)$] if $i \in \{ 1, \cdots , p \},$ then we will have $w_i=v_i^c.$ Since $w_i \le w,$ we will have that $w^c \le w_i^c =(v_i^c)^c=v_i;$ since $v_i \in (A')^{-1}(N)$ and $A'$ is down-negative, it follows that $w^c \in (A')^{-1}(N).$ Therefore since $A'$ is complemented positive, $w=(w^c)^c \in (A')^{-1} (P),$ i.e. $A'(w)=P,$ which is a contradiction.
\end{itemize}
This shows that $Core_{w+}(A)$ is a core for $A$ on $X$.

Let now $W$ be a core for $A$ on $X$. At first we observe that

 $$Core_{w+}(A)=[Min(X^A_P) \setminus (Max(X^A_N))^c] \cup^d Max(X^A_N).$$

Moreover, by Proposition \ref{minincl}-$i)$ we have that
\begin{equation} \label{eq1}
Min(X^A_P) \subseteq W^A_P \cup (W^A_N)^c,
\end{equation}
and by Corollary \ref{maxneg}-$i)$ we have that
\begin{equation} \label{eq2}
Max(X^A_N) \subseteq W^A_N.
\end{equation}
Therefore, to show that $Core_{w+}(A) \subseteq W$, by (\ref{eq2}) it is sufficient to prove that
$$H=Min(X^A_P) \setminus (Max(X^A_N)^c) \subseteq W^A_P.$$
Since $H \subseteq Min(X^A_P),$ by (\ref{eq1}) it is sufficient to prove that $H \cap (W^A_N)^c = \emptyset.$ Suppose on the contrary that
there exists $w \in H$ such that $w \in (W^A_N)^c.$ In this case there exists $\tilde{w} \in W^A_N$ such that $\tilde{w}^c=w.$ It follows that $w$ is a minimal positive of $A$ ($w \in H$) such that $w^c=\tilde{w}$ is negative for $A,$ hence by Proposition \ref{minimax}-$i)$ it follows that $\tilde{w} \in Max(X^A_N)$ and therefore $w=\tilde{w}^c \in Max(X^A_N)^c,$ which is a contradiction since $w \in H.$

$ii)$ Similarly using Proposition \ref{primoNucleo}, Proposition \ref{minincl}-$ii)$, Corollary \ref{maxneg}-$ii)$ and Proposition \ref{minimax}-$ii)$
\end{proof}

\section{Essential properties of a $\mathcal{W}_{\pm}(X,{\bf 2})$-fundamental core}

In this section we determine the properties characterizing the fundamental core of a WBTM. This will lead us to define the concepts of $w_+-${\it basis} and $w_{-}-${\it basis} for $X$. At the end of the section we will show that each $w_+-$basis identifies uniquely the fundamental core of a +WBTM on $X$ and each $w_--$basis identifies uniquely the fundamental core of a -WBTM on $X$.
In all this section, $X$ will denote a finite SIP.
\begin{definition}
i) A $w_+-$basis for $X$ is an ordered couple $\langle Y_+ | Y_{-} \rangle$, where $Y_{+}$ and $Y_{-}$ are two disjoint anti-chains of $X$ such that:
\begin{itemize}
\item[B1+)] $(\downarrow Y_{+}) \cap (Y_{-}^c) = \emptyset;$ \\
\item[B2+)] $(\uparrow Y_{+} \cup \uparrow (Y_{-})^c) \cap \downarrow Y_{-} = \emptyset;$ \\
\item[B3+)] $X=(\uparrow Y_{+} \cup \uparrow (Y_{-})^c) \cup \downarrow Y_{-}.$\\
\end{itemize}

ii) A $w_--$basis for $X$ is an ordered couple $\langle Y_+ | Y_{-} \rangle$, where $Y_{+}$ and $Y_{-}$ are two disjoint anti-chains of $X$ such that:
\begin{itemize}
\item[B1-)] $(\uparrow Y_{-}) \cap (Y_{+}^c) = \emptyset;$ \\
\item[B2-)] $(\downarrow Y_{-} \cup \downarrow (Y_{+})^c) \cap \uparrow Y_{+} = \emptyset;$ \\
\item[B3-)] $X=(\downarrow Y_{-} \cup \downarrow (Y_{+})^c) \cup \uparrow Y_{+}.$\\
\end{itemize}
\end{definition}

Two $w_+$-bases [$w_-$-bases] $\langle Y_+ | Y_{-} \rangle$ and $\langle Y'_+ | Y'_{-} \rangle$ are considered equal if $Y_+ = Y'_+$ and $Y_- = Y'_-$.

\begin{proposition}\label{Ker(A)basis}
i) If $A\in \mathcal{W}_+(X,{\bf 2})$  and if $W=Core_{w+}(A)$, then $\langle W^A_P | W^A_N \rangle$ is a $w_+-$basis for $X$.

ii) If $A\in \mathcal{W}_-(X,{\bf 2})$  and if $W=Core_{w-}(A)$, then $\langle W^A_P | W^A_N \rangle$ is a $w_--$basis for $X$.
\end{proposition}
\begin{proof}
$i)$ By definition of $Core_{w+}(A)$, we have that $W^A_P = Min(X^A_P) \setminus (Max(X^A_N))^c$ and $W^A_N = Max(X^A_N)$.
By Theorem \ref{NA}-$i)$ we know that $W$ is a core for $A$, therefore, by Theorem \ref{generators}-$i)$, we have
$X=(\uparrow W_P^A \cup \uparrow(W_N^A)^c) \cup^d \downarrow W^A_N$.

Moreover, since the elements of $W^A_P$ are a part of the minimal positives of $A$ and the elements of $W^A_N$ are all the maximal negatives of $A$, it follows that $W^A_P$ and $W^A_N$ are two disjoint anti-chains of $X$. It will remain to prove that $\downarrow W^A_P \cap (W^A_N)^c = \emptyset$.
Since $W^A_P = Min(X^A_P) \setminus (W^A_N)^c$, then $W^A_P \cap (W^A_N)^c = \emptyset$. Let us suppose now by contradiction that there exists an element $z \in \downarrow W^A_P \cap (W^A_N)^c$. This implies the existence of an element $x \in W^A_P$ such that $z \le x$. Since $W^A_P \cap (W^A_N)^c = \emptyset$, we have that $z < x$ (if $z=x$, then $x \in W^A_P \cap (W^A_N)^c = \emptyset$). Since $x$ is a minimal positive of $A$ and $A(z)=P$ (because $z \in (W^A_N)^c$), this is a contradiction.

$ii)$ Similar arguments apply using Theorem \ref{NA}-$ii)$ and Theorem \ref{generators}-$ii)$.

\end{proof}

The following proposition will be essential in \cite{chias-mar-nardi}:
\begin{proposition}\label{basisKer(A)}
$i)$ Let $\langle W_+ | W_{-} \rangle$ be a $w_+-$basis for $X$. If we set $W = W_+ \cup^d W_ -$ and
$$A(x)=\left\{ \begin{array}{lll}
                 P & \textrm{if} & x \in \uparrow (W_+) \cup \uparrow (W_{-}^c) \\
                 N & \textrm{if} & x \in \downarrow W_{-}
                 \end{array} \right.$$
then $A$ is a +WBTM on $X$ and $W = Core_{w+}(A)$.

ii) Let $\langle W_+ | W_{-} \rangle$ be a $w_--$basis for $X$. If we set $W = W_+ \cup^d W_ -$ and
$$A(x)=\left\{ \begin{array}{lll}
                 P & \textrm{if} & x \in \uparrow (W_+) \\
                 N & \textrm{if} & x \in \downarrow W_{-} \cup \downarrow (W_{+}^c)
                 \end{array} \right.$$
then $A$ is a -WBTM on $X$ and $W = Core_{w-}(A)$.
\end{proposition}
\begin{proof}
$i)$ Let us observe that $A$ is well defined because $( \uparrow (W_+) \cup \uparrow (W_{-}^c)) \cap \downarrow W_{-} = \emptyset$; moreover $W_+ = W^A_P$, $W_- = W^A_N$ by definition of $A$.
Now we prove that $A$ is up-positive and complemented positive. Indeed, since
$A^{-1}(P)=\uparrow W_{+} \cup \uparrow W_{-}^c$ is a union of two up-sets, it is also an up-set. Furthermore, since $(A^{-1}(N))^c=(\downarrow W_{-})^c=\uparrow(W^c_{-}) \subseteq A^{-1} (P),$ it follows that $A$ is complemented positive. Hence $A$ is a +WBTM.
Suppose now that $B$ is another +WBTM on $X$ such that $B_{|W}=A_{|W}.$ We need to prove that $B=A$ on all $X.$ Suppose on the contrary that there exists $w \in X$ such that $B(w) \neq A(w).$
\begin{itemize}
\item[1)] Suppose that $w \in \uparrow W_{+}.$ In this case, $A(w)=P,$ hence it holds $B(w)=N.$ Since $w\in \uparrow W_{+},$ there exists $\tilde{w} \in W_{+}$ such that $\tilde{w} \le w;$ by hypothesis $B$ and $A$ coincides on $W$ and hence $B(\tilde{w})=A(\tilde{w})=P.$ Since $B$ is a +WBTM, it follows that $B$ is up-positive and hence (since $\tilde{w} \le w$)
we have that $B(w)=P$ and this is a contradiction.\\
\item[2)] Suppose that $w \in \downarrow W_{-}.$ In this case, $A (w)=N,$ and hence $B(w)=P.$ Since $w\in \downarrow W_{-},$ there exists $\tilde{w} \in W_{-}$ such that $w \le \tilde{w};$ by hypothesis $B(\tilde{w})=A(\tilde{w})=N.$ Since $B$ is down-negative, with $w \le \tilde{w},$ we have that $B(w)=N.$ This is a contradiction.\\
\item[3)] Suppose that $w \in \uparrow(W^c_{-}).$ In this case, $A(w)=P,$ and hence $B(w)=N.$ Since $w \in \uparrow(W^c_{-}),$
there exists $\tilde{w} \in W^c_{-}$ such that $\tilde{w} \le w.$ Since $\tilde{w} \in W^c_{-},$ there exists $\overline{w} \in W_{-}$ such that $\tilde{w}=\overline{w}^c,$ and hence $B(\overline{w})=A(\overline{w})=N.$ Since $A$ and $B$ are complemented positive, it follows that $B(\overline{w}^c)=A(\overline{w}^c)=P,$ i.e. $B(\tilde{w})=P.$ Now since $\tilde{w} \le w$ and $B$ is up-positive, it follows that $B(w)=P,$ and this is a contradiction.
\end{itemize}
\medskip
Hence $W$ is a core for $A$. By Theorem \ref{NA}-$i)$ it follows then that $Core_w(A) \subseteq W$.

We prove now that \begin{equation} \label{WMin}
W^A_P \subseteq Min(X^A_P).
\end{equation}

Suppose that $w \in W^A_P$ and that by contradiction $w \not\in Min(X^A_P)$.

 Then there exists $\overline{w} \in X^A_P$ such that $\overline{w} < w.$ By Corollary \ref{posug} $i)$ we have that $\overline{w} \in \uparrow W^A_P$ or $\overline{w} \in \uparrow (W_N^A)^c.$

$1)$ if $\overline{w} \in \uparrow W^A_P,$ there exists $\tilde{w} \in W^A_P$ such that $\tilde{w} \le \overline{w}$ and hence we will have that $\tilde{w} < \overline{w} \le w,$ with $\tilde{w},w \in W^A_P,$ against the hypothesis that $W^A_P$ is an anti-chain.

$2)$ if $\overline{w} \in \uparrow (W^A_N)^c,$ there exists $\tilde{w} \in (W_N^A)^c$ such that $\tilde{w} \le \overline{w}$ and hence $\tilde{w} \le \overline{w} < w,$ with $\tilde{w} \in (W^A_N)^c$ and $w \in W^A_P,$ against the $B1+)$ and the hypothesis that $\langle W^A_P | W^A_N \rangle$ is a $w-$basis for $X$.

This proves (\ref{WMin}).

Let us suppose now that $Core_{w+}(A) \neq W$. Since $Core_{w+}(A) \subseteq W$, this implies that $Core_{w+}(A) \varsubsetneqq W$, and hence that $|Core_{w+}(A)|<|W|$.
Let $\tilde{W}=Core_{w+}(A).$

Since $W$ and $\tilde{W}$ are both two cores for $A$, by Proposition \ref{minmin}-$i)$ we know that
$$Min(W^A_P \cup^d (W^A_N)^c)=Min(X^A_P)=Min(\tilde{W}^A_P \cup (\tilde{W}^A_N)^c).$$
Since $\langle W^A_P | W^A_N \rangle$ is a $w-$basis for $X$, we have that $W^A_N$ is an anti-chain, moreover, by definition of $Core_{w+}(A)$, also $\tilde{W}^A_N$ is an anti-chain; by Corollary \ref{maxneg2} i), then it follows that $W^A_N=Max(X^A_N)=\tilde{W}^A_N$, and hence $(W^A_N)^c=(\tilde{W}^A_N)^c.$ Since $W=W^A_P \cup^d W^A_N$ and $\tilde{W}=\tilde{W}^A_P \cup^d \tilde{W}^A_N,$ by the equality $W^A_N=\tilde{W}^A_N$ and by the inequality $|\tilde{W}|< |W|,$ it follows that
\begin{equation} \label{*1}
|\tilde{W}^A_P|<|W^A_P|.
\end{equation}
By (\ref{WMin}) we have that
\begin{equation}\label{*2}
W^A_P \subseteq Min(X^A_P)=Min(W^A_P \cup^d (W^A_N)^c)=Min(\tilde{W}^A_P \cup (\tilde{W}^A_N)^c).
\end{equation}
Since $W^A_P \cap (\tilde{W}^A_N)^c = W^A_P \cap (W^A_N)^c = \emptyset$ in view of the fact that $\langle W^A_P | W^A_N \rangle$ is a $w_+-$basis for $X$, by (\ref{*2}) it follows that $W^A_P \subseteq \tilde{W}^A_P$, and hence
$|W^A_P| \le |\tilde{W}^A_P|$, that is in contradiction with (\ref{*1}). This proves that $W=Core_{w+}(A)$.

$ii)$ The same reasoning applies using Theorem \ref{NA}-$ii)$, Corollary \ref{posug}-$ii)$ and Proposition \ref{minmin}-$ii)$.

\end{proof}

We denote now with $\mathcal{B}_{w+}(X)$ the family of all $w_+-$bases on $X$ and with $\mathcal{B}_{w-}(X)$ the family of all $w_--$bases on $X$. If $A \in \mathcal{W}_+(X,{\bf 2})$, by Proposition \ref{Ker(A)basis}-$i),$ it follows that
$\langle W^A_P | W^A_N \rangle \in \mathcal{B}_{w+}(X)$, where $W = Core_{w+}(A)$. This defines an application $h_+ : \mathcal{W}_+(X,{\bf 2}) \to \mathcal{B}_{w+}(X)$ such that $h_+(A) = \langle W^A_P | W^A_N \rangle$, where $W = Core_{w+}(A)$.

If $A \in \mathcal{W}_-(X,{\bf 2})$, by Proposition \ref{Ker(A)basis}-$ii)$ we can define a similar map $h_- : \mathcal{W}_-(X,{\bf 2}) \to \mathcal{B}_{w-}(X).$
It holds then the following result.

\begin{theorem}\label{gBiunivoca}
The maps $h_+$ and $h_-$ are bijective.
\end{theorem}
\begin{proof}
The map $h_+$ is onto by virtue of Proposition \ref{basisKer(A)}-$i)$. We prove now that $h_+$ is a one-to-one map.

Let $A$ and $B$ be two +WBTM's on $X$ such that $\langle W^A_P | W^A_N \rangle = \langle \tilde{W}^B_P | \tilde{W}^B_N \rangle$, where $W = Core_{w+}(A)$ and $\tilde{W} = Core_{w+}(B)$.
This means that $W^A_P=\tilde{W}^B_P$ and $W^A_N=\tilde{W}^B_N$.
Then, if $w \in \tilde{W}^B_P,$ we have that $B(w) = P$ and also $A(w) = P$; similarly, if $w\in \tilde{W}^B_N$, we have that $B(w)=N$ and $A(w)=N$.
Therefore $A_{|_{\tilde{W}}} = B_{|_{\tilde{W}}}$. Since $\tilde{W}$ is a core for $B$, it follows that $A=B$.

The case of $h_-$ is analogue.
\end{proof}

\section{Applications to the $(n,r)$-systems}

In this section we apply the previous general results to the case $X=S(n,r)$. We set

$$\mathcal{OP}(n,r)=\mathcal{OP}(S(n,r),{\bf 2}) \cap \mathcal{BT}(n,r),$$

$$\mathcal{W_+}(n,r)=\{A\in \mathcal{W_+}(S(n,r),{\bf 2}) \cap \mathcal{BT}(n,r)\ : A(r\cdots21|12\cdots(n-r)) = P\},$$

$$\mathcal{W_-}(n,r)=\{A\in \mathcal{W_-}(S(n,r),{\bf 2}) \cap \mathcal{BT}(n,r)\ : A(r\cdots21|12\cdots(n-r)) = N\}.$$

The family $\mathcal{W_+}(n,r)$ satisfies the hypotheses of p.l.c. and the family $\mathcal{W_-}(n,r)$ satisfies the hypotheses of n.l.c., i.e:

$$\chi(W_+CTSyst(n,r)) \subseteq \mathcal{W_+}(n,r) {\rm \,\,and\,\,}  \mathcal{W_+}(n,r) \cap \chi(W_-CTSyst(n,r)) = \emptyset,$$

$$\chi(W_-CTSyst(n,r)) \subseteq \mathcal{W_-}(n,r) {\rm \,\,and\,\,}  \mathcal{W_-}(n,r) \cap \chi(W_+CTSyst(n,r)) = \emptyset.$$

We can then apply the local criteria to the previous two families of boolean total maps on $S(n,r)$. If we apply the p.l.c. to a map $A\in \mathcal{W_+}(n,r)$, we take the $\mathcal{W_+}(S(n,r),{\bf 2})$-fundamental core, that is obviously also a $\mathcal{W_+}(n,r)$-core. Similarly, if we apply the n.l.c. to a map $A\in \mathcal{W_-}(n,r)$, we take the $\mathcal{W_-}(S(n,r),{\bf 2})$-fundamental core, that is  also a $\mathcal{W_-}(n,r)$-core. In these cases we say simply ``the fundamental core" of $A$.

In the sequel, for semplicity, we will write a partial map as the set of the strings of its domain followed by an $N$ if they are negative or a $P$ if they are positive.
\begin{example}
Let us consider the case $n=5$ and $r=3$. We take the system $\mathcal{S}\in W_-TSyst(5,3)$ such that the relative boolean map $A_{\mathcal{S}}$ associated to it is the following
(the green nodes are P and the red nodes are N):

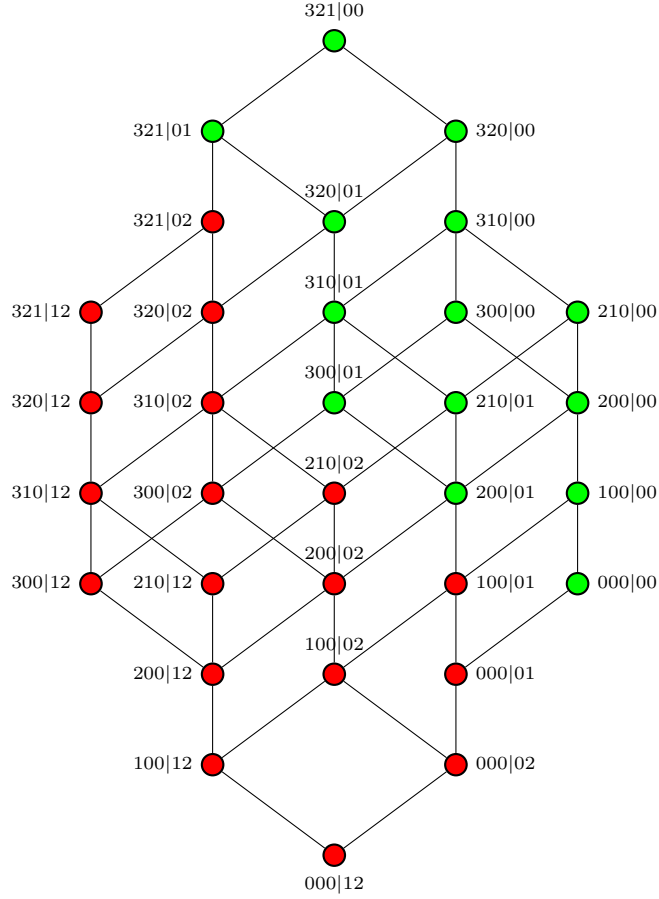
\begin{center}

\begin{tikzpicture}
 [inner sep=1.0mm,
 placeg/.style={circle,draw=black!100,fill=green!100,thick},
 placer/.style={circle,draw=black!100,fill=red!100,thick},scale=0.4]

 \path
{(0,0)node(1a) [placer,label=270:{\footnotesize$000|12$}]{}}

{ (-4,3)node(1b) [placer,label=180:{\footnotesize$100|12$}]{}}
{ (4,3)node(2b) [placer,label=0:{\footnotesize$000|02$}]{}}

{ (-4,6)node(1c) [placer,label=180:{\footnotesize$200|12$}]{}}
{ (0,6)node(2c) [placer,label=90:{\footnotesize$100|02$}]{}}
{  (4,6)node(3c) [placer,label=0:{\footnotesize$000|01$}]{}}

{ (-8,9)node(1d) [placer,label=180:{\footnotesize$300|12$}]{}}
{ (-4,9)node(2d) [placer,label=180:{\footnotesize$210|12$}]{}}
{ (0,9)node(3d) [placer,label=90:{\footnotesize$200|02$}]{}}
{ (4,9)node(4d) [placer,label=0:{\footnotesize$100|01$}]{}}
{ (8,9)node(5d) [placeg,label=0:{\footnotesize$000|00$}]{}}

{ (-8,12)node(1e) [placer,label=180:{\footnotesize$310|12$}]{}}
{ (-4,12)node(2e) [placer,label=180:{\footnotesize$300|02$}]{}}
{ (0,12)node(3e) [placer,label=90:{\footnotesize$210|02$}]{}}
{ (4,12)node(4e) [placeg,label=0:{\footnotesize$200|01$}]{}}
{ (8,12)node(5e) [placeg,label=0:{\footnotesize$100|00$}]{}}

{ (-8,15)node(1f) [placer,label=180:{\footnotesize$320|12$}]{}}
{ (-4,15)node(2f) [placer,label=180:{\footnotesize$310|02$}]{}}
{ (0,15)node(3f) [placeg,label=90:{\footnotesize$300|01$}]{}}
{ (4,15)node(4f) [placeg,label=0:{\footnotesize$210|01$}]{}}
{ (8,15)node(5f) [placeg,label=0:{\footnotesize$200|00$}]{}}

{ (-8,18)node(1g) [placer,label=180:{\footnotesize$321|12$}]{}}
{ (-4,18)node(2g) [placer,label=180:{\footnotesize$320|02$}]{}}
{ (0,18)node(3g) [placeg,label=90:{\footnotesize$310|01$}]{}}
{ (4,18)node(4g) [placeg,label=0:{\footnotesize$300|00$}]{}}
{ (8,18)node(5g) [placeg,label=0:{\footnotesize$210|00$}]{}}

{  (-4,21)node(1h) [placer,label=180:{\footnotesize$321|02$}]{}}
{ (0,21)node(2h) [placeg,label=90:{\footnotesize$320|01$}]{}}
{(4,21)node(3h) [placeg,label=0:{\footnotesize$310|00$}]{}}

{ (-4,24)node(1i) [placeg,label=180:{\footnotesize$321|01$}]{}}
{ (4,24)node(2i) [placeg,label=0:{\footnotesize$320|00$}]{}}

{ (0,27)node(1j) [placeg,label=90:{\footnotesize$321|00$}]{}};

 \draw (1a)--(1b); 
 \draw (1a)--(2b);

 \draw (1b)--(1c); 
 \draw (1b)--(2c);
 \draw (2b)--(2c);
 \draw (2b)--(3c);

 \draw (1c)--(1d); 
 \draw (1c)--(2d);
 \draw (1c)--(3d);
 \draw (2c)--(3d);
 \draw (2c)--(4d);
 \draw (3c)--(4d);
 \draw (3c)--(5d);

 \draw (1d)--(1e); 
 \draw (1d)--(2e);
 \draw (2d)--(1e);
 \draw (2d)--(3e);
 \draw (3d)--(2e);
 \draw (3d)--(3e);
 \draw (3d)--(4e);
 \draw (4d)--(4e);
 \draw (4d)--(5e);
 \draw (5d)--(5e);

 \draw (1e)--(1f); 
 \draw (1e)--(2f);
 \draw (2e)--(2f);
 \draw (2e)--(3f);
 \draw (3e)--(2f);
 \draw (3e)--(4f);
 \draw (4e)--(3f);
 \draw (4e)--(4f);
 \draw (4e)--(5f);
 \draw (5e)--(5f);

 \draw (1f)--(1g); 
 \draw (1f)--(2g);
 \draw (2f)--(2g);
 \draw (2f)--(3g);
 \draw (3f)--(3g);
 \draw (3f)--(4g);
 \draw (4f)--(3g);
 \draw (4f)--(5g);
 \draw (5f)--(4g);
 \draw (5f)--(5g);

 \draw (1g)--(1h); 
 \draw (2g)--(1h);
 \draw (2g)--(2h);
 \draw (3g)--(2h);
 \draw (3g)--(3h);
 \draw (4g)--(3h);
 \draw (5g)--(3h);

 \draw (1h)--(1i); 
 \draw (2h)--(1i);
 \draw (2h)--(2i);
 \draw (3h)--(2i);

 \draw (1i)--(1j); 
 \draw (2i)--(1j);

 \end{tikzpicture}

\end{center}

It results that $A_{\mathcal{S}}\in \mathcal{W_-}(5,3)$. It is easy to verify that the fundamental core of $A_{\mathcal{S}}$ is the partial map $B=\{ 321|02N, 100|01N, 000|00P, 200|01P \}$. The $(5,3)$-system $\mathcal{S}_B$ is then the following:

$$
\mathcal{S}_B:\left\{ \begin{array}{l}
                    x_3 \ge x_2 \ge x_1 \ge 0 > x_4 \ge x_5  \\
                    x_1 + x_2 + x_3 + x_5 < 0 \\
                    x_1 + x_4 < 0 \\
                    x_2 + x_4 \ge 0
                    \end{array}
                   \right.
$$

A solution of this system is easily given by:
$$x_3=\frac{1}{2}, x_2=\frac{1}{3}, x_1=\frac{1}{6}, x_4=-\frac{1}{5}, x_5=-\frac{6}{5}.$$
Then, by n.l.c. it follows that $\mathcal{S}\in W_-CTSyst(5,3)$, i.e. it is compatible and has the same solutions of $\mathcal{S}_B$.

\end{example}

\begin{example}
Let us consider again the case $n=5$ and $r=3$. We take the system $\mathcal{S}\in W_+TSyst(5,3)$ such that the relative boolean map $A_{\mathcal{S}}$ associated to it is the following
(as before, the green nodes are P and the red nodes are N):

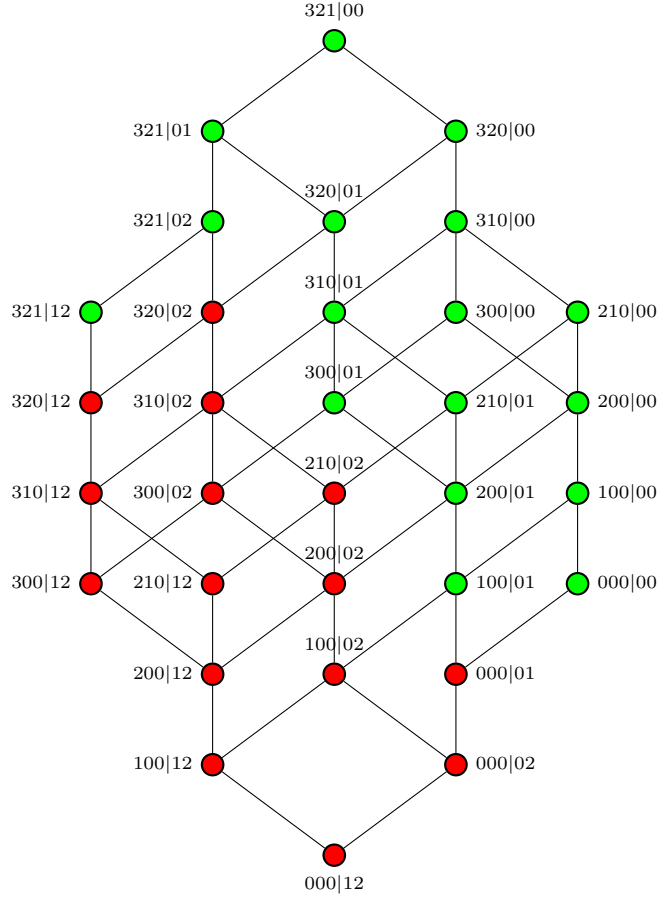
\begin{center}

\begin{tikzpicture}
 [inner sep=1.0mm,
 placeg/.style={circle,draw=black!100,fill=green!100,thick},
 placer/.style={circle,draw=black!100,fill=red!100,thick},scale=0.4]

 \path
{ (0,0)node(1a) [placer,label=270:{\footnotesize$000|12$}]{}}

{ (-4,3)node(1b) [placer,label=180:{\footnotesize$100|12$}]{}}
{(4,3)node(2b) [placer,label=0:{\footnotesize$000|02$}]{}}

{ (-4,6)node(1c) [placer,label=180:{\footnotesize$200|12$}]{}}
{ (0,6)node(2c) [placer,label=90:{\footnotesize$100|02$}]{}}
{  (4,6)node(3c) [placer,label=0:{\footnotesize$000|01$}]{}}

{ (-8,9)node(1d) [placer,label=180:{\footnotesize$300|12$}]{}}
{ (-4,9)node(2d) [placer,label=180:{\footnotesize$210|12$}]{}}
{ (0,9)node(3d) [placer,label=90:{\footnotesize$200|02$}]{}}
{ (4,9)node(4d) [placeg,label=0:{\footnotesize$100|01$}]{}}
{ (8,9)node(5d) [placeg,label=0:{\footnotesize$000|00$}]{}}

{ (-8,12)node(1e) [placer,label=180:{\footnotesize$310|12$}]{}}
{ (-4,12)node(2e) [placer,label=180:{\footnotesize$300|02$}]{}}
{ (0,12)node(3e) [placer,label=90:{\footnotesize$210|02$}]{}}
{ (4,12)node(4e) [placeg,label=0:{\footnotesize$200|01$}]{}}
{ (8,12)node(5e) [placeg,label=0:{\footnotesize$100|00$}]{}}

{ (-8,15)node(1f) [placer,label=180:{\footnotesize$320|12$}]{}}
{ (-4,15)node(2f) [placer,label=180:{\footnotesize$310|02$}]{}}
{ (0,15)node(3f) [placeg,label=90:{\footnotesize$300|01$}]{}}
{ (4,15)node(4f) [placeg,label=0:{\footnotesize$210|01$}]{}}
{ (8,15)node(5f) [placeg,label=0:{\footnotesize$200|00$}]{}}

{ (-8,18)node(1g) [placeg,label=180:{\footnotesize$321|12$}]{}}
{ (-4,18)node(2g) [placer,label=180:{\footnotesize$320|02$}]{}}
{ (0,18)node(3g) [placeg,label=90:{\footnotesize$310|01$}]{}}
{ (4,18)node(4g) [placeg,label=0:{\footnotesize$300|00$}]{}}
{ (8,18)node(5g) [placeg,label=0:{\footnotesize$210|00$}]{}}

{  (-4,21)node(1h) [placeg,label=180:{\footnotesize$321|02$}]{}}
{ (0,21)node(2h) [placeg,label=90:{\footnotesize$320|01$}]{}}
{(4,21)node(3h) [placeg,label=0:{\footnotesize$310|00$}]{}}

{ (-4,24)node(1i) [placeg,label=180:{\footnotesize$321|01$}]{}}
{ (4,24)node(2i) [placeg,label=0:{\footnotesize$320|00$}]{}}

{ (0,27)node(1j) [placeg,label=90:{\footnotesize$321|00$}]{}};

 \draw (1a)--(1b); 
 \draw (1a)--(2b);

 \draw (1b)--(1c); 
 \draw (1b)--(2c);
 \draw (2b)--(2c);
 \draw (2b)--(3c);

 \draw (1c)--(1d); 
 \draw (1c)--(2d);
 \draw (1c)--(3d);
 \draw (2c)--(3d);
 \draw (2c)--(4d);
 \draw (3c)--(4d);
 \draw (3c)--(5d);

 \draw (1d)--(1e); 
 \draw (1d)--(2e);
 \draw (2d)--(1e);
 \draw (2d)--(3e);
 \draw (3d)--(2e);
 \draw (3d)--(3e);
 \draw (3d)--(4e);
 \draw (4d)--(4e);
 \draw (4d)--(5e);
 \draw (5d)--(5e);

 \draw (1e)--(1f); 
 \draw (1e)--(2f);
 \draw (2e)--(2f);
 \draw (2e)--(3f);
 \draw (3e)--(2f);
 \draw (3e)--(4f);
 \draw (4e)--(3f);
 \draw (4e)--(4f);
 \draw (4e)--(5f);
 \draw (5e)--(5f);

 \draw (1f)--(1g); 
 \draw (1f)--(2g);
 \draw (2f)--(2g);
 \draw (2f)--(3g);
 \draw (3f)--(3g);
 \draw (3f)--(4g);
 \draw (4f)--(3g);
 \draw (4f)--(5g);
 \draw (5f)--(4g);
 \draw (5f)--(5g);

 \draw (1g)--(1h); 
 \draw (2g)--(1h);
 \draw (2g)--(2h);
 \draw (3g)--(2h);
 \draw (3g)--(3h);
 \draw (4g)--(3h);
 \draw (5g)--(3h);

 \draw (1h)--(1i); 
 \draw (2h)--(1i);
 \draw (2h)--(2i);
 \draw (3h)--(2i);

 \draw (1i)--(1j); 
 \draw (2i)--(1j);

 \end{tikzpicture}

\end{center}

It is easy to verify that the fundamental core of $A_{\mathcal{S}}$ is the partial map\\
 $B=\{ 320|02N, 321|12P, 000|00P, 000|01N \}$. Therefore the system $\mathcal{S}_B$ is the following :

$$
\left\{ \begin{array}{l}
                    x_3 \ge x_2 \ge x_1 \ge 0 > x_4 \ge x_5  \\
                    x_1 + x_2 + x_3 + x_4 + x_5 \ge 0 \\
                    x_3 + x_2 + x_5 < 0
                    \end{array}
                   \right.
$$

A solution of this system is easily given by:
$$x_3=1, x_2=1, x_1=0.9, x_4=-0.8, x_5=-2.1.$$

By p.l.c. it follows then that $\mathcal{S}\in W_+CTSyst(5,3)$, i.e. $\mathcal{S}$ is compatible and equivalent to $\mathcal{S}_B$.

\end{example}

In the previous example, we showed two different $(5,3)$-total systems both compatible. In the next example, we show a $(6,3)$-total system $\mathcal{S}$ that is not compatible but such that $A_{\mathcal{S}}\in \mathcal{W}_+(6,3)$. Hence the next example shows that the inclusion $\chi(W_+CTSyst(n,r)) \subseteq \mathcal{W_+}(n,r)$ is strict, i.e. there exist maps in $\mathcal{W_+}(n,r)$ whose associated $(n,r)$-system has not solutions.

\begin{example}
Let us consider the following map $A\in \mathcal{BT}(6,3)$ (the green nodes are P and the red nodes are N):

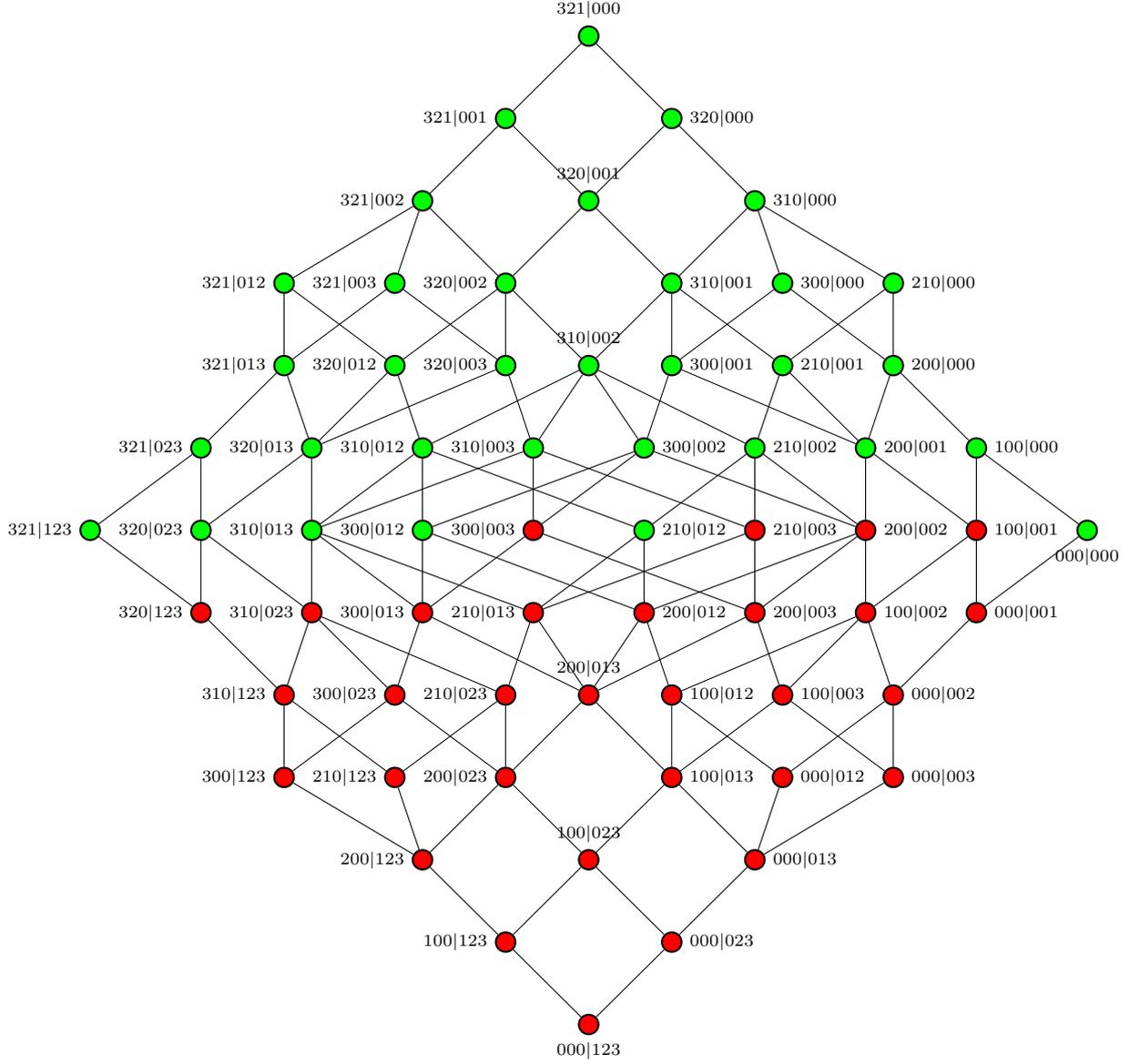
\begin{center}

\begin{tikzpicture}
 [inner sep=1.0mm,
 placeg/.style={circle,draw=black!100,fill=green!100,thick},
 placer/.style={circle,draw=black!100,fill=red!100,thick},
 scale=0.4]

 \path
 (-24,0)node(1a) [placer,label=270:{\footnotesize$000|123$}]{}

 (-27,3)node(1b) [placer,label=180:{\footnotesize$100|123$}]{}
 (-21,3)node(2b) [placer,label=0:{\footnotesize$000|023$}]{}

 (-30,6)node(1c) [placer,label=180:{\footnotesize$200|123$}]{}
 (-24,6)node(2c) [placer,label=90:{\footnotesize$100|023$}]{}
 (-18,6)node(3c) [placer,label=0:{\footnotesize$000|013$}]{}

 (-35,9)node(1d) [placer,label=180:{\footnotesize$300|123$}]{}
 (-31,9)node(2d) [placer,label=180:{\footnotesize$210|123$}]{}
 (-27,9)node(3d) [placer,label=180:{\footnotesize$200|023$}]{}
 (-21,9)node(4d) [placer,label=0:{\footnotesize$100|013$}]{}
 (-17,9)node(5d) [placer,label=0:{\footnotesize$000|012$}]{}
 (-13,9)node(6d) [placer,label=0:{\footnotesize$000|003$}]{}

 (-35,12)node(1e) [placer,label=180:{\footnotesize$310|123$}]{}
 (-31,12)node(2e) [placer,label=180:{\footnotesize$300|023$}]{}
 (-27,12)node(3e) [placer,label=180:{\footnotesize$210|023$}]{}
 (-24,12)node(4e) [placer,label=90:{\footnotesize$200|013$}]{}
 (-21,12)node(5e) [placer,label=0:{\footnotesize$100|012$}]{}
 (-17,12)node(6e) [placer,label=0:{\footnotesize$100|003$}]{}
 (-13,12)node(7e) [placer,label=0:{\footnotesize$000|002$}]{}

 (-38,15)node(1f) [placer,label=180:{\footnotesize$320|123$}]{}
 (-34,15)node(2f) [placer,label=180:{\footnotesize$310|023$}]{}
 (-30,15)node(3f) [placer,label=180:{\footnotesize$300|013$}]{}
 (-26,15)node(4f) [placer,label=180:{\footnotesize$210|013$}]{}
 (-22,15)node(5f) [placer,label=0:{\footnotesize$200|012$}]{}
 (-18,15)node(6f) [placer,label=0:{\footnotesize$200|003$}]{}
 (-14,15)node(7f) [placer,label=0:{\footnotesize$100|002$}]{}
 (-10,15)node(8f) [placer,label=0:{\footnotesize$000|001$}]{}

 (-42,18)node(1g) [placeg,label=180:{\footnotesize$321|123$}]{}
 (-38,18)node(2g) [placeg,label=180:{\footnotesize$320|023$}]{}
 (-34,18)node(3g) [placeg,label=180:{\footnotesize$310|013$}]{}
 (-30,18)node(4g) [placeg,label=180:{\footnotesize$300|012$}]{}
 (-26,18)node(5g) [placer,label=180:{\footnotesize$300|003$}]{}
 (-22,18)node(6g) [placeg,label=0:{\footnotesize$210|012$}]{}
 (-18,18)node(7g) [placer,label=0:{\footnotesize$210|003$}]{}
 (-14,18)node(8g) [placer,label=0:{\footnotesize$200|002$}]{}
 (-10,18)node(9g) [placer,label=0:{\footnotesize$100|001$}]{}
 (-6,18)node(10g) [placeg,label=270:{\footnotesize$000|000$}]{}

 (-38,21)node(1h) [placeg,label=180:{\footnotesize$321|023$}]{}
 (-34,21)node(2h) [placeg,label=180:{\footnotesize$320|013$}]{}
 (-30,21)node(3h) [placeg,label=180:{\footnotesize$310|012$}]{}
 (-26,21)node(4h) [placeg,label=180:{\footnotesize$310|003$}]{}
 (-22,21)node(5h) [placeg,label=0:{\footnotesize$300|002$}]{}
 (-18,21)node(6h) [placeg,label=0:{\footnotesize$210|002$}]{}
 (-14,21)node(7h) [placeg,label=0:{\footnotesize$200|001$}]{}
 (-10,21)node(8h) [placeg,label=0:{\footnotesize$100|000$}]{}

 (-35,24)node(1i) [placeg,label=180:{\footnotesize$321|013$}]{}
 (-31,24)node(2i) [placeg,label=180:{\footnotesize$320|012$}]{}
 (-27,24)node(3i) [placeg,label=180:{\footnotesize$320|003$}]{}
 (-24,24)node(4i) [placeg,label=90:{\footnotesize$310|002$}]{}
 (-21,24)node(5i) [placeg,label=0:{\footnotesize$300|001$}]{}
 (-17,24)node(6i) [placeg,label=0:{\footnotesize$210|001$}]{}
 (-13,24)node(7i) [placeg,label=0:{\footnotesize$200|000$}]{}

 (-35,27)node(1j) [placeg,label=180:{\footnotesize$321|012$}]{}
 (-31,27)node(2j) [placeg,label=180:{\footnotesize$321|003$}]{}
 (-27,27)node(3j) [placeg,label=180:{\footnotesize$320|002$}]{}
 (-21,27)node(4j) [placeg,label=0:{\footnotesize$310|001$}]{}
 (-17,27)node(5j) [placeg,label=0:{\footnotesize$300|000$}]{}
 (-13,27)node(6j) [placeg,label=0:{\footnotesize$210|000$}]{}

 (-30,30)node(1k) [placeg,label=180:{\footnotesize$321|002$}]{}
 (-24,30)node(2k) [placeg,label=90:{\footnotesize$320|001$}]{}
 (-18,30)node(3k) [placeg,label=0:{\footnotesize$310|000$}]{}

 (-27,33)node(1l) [placeg,label=180:{\footnotesize$321|001$}]{}
 (-21,33)node(2l) [placeg,label=0:{\footnotesize$320|000$}]{}

 (-24,36)node(1m) [placeg,label=90:{\footnotesize$321|000$}]{};

 \draw (1a)--(1b); 
 \draw (1a)--(2b);

 \draw (1b)--(1c); 
 \draw (1b)--(2c);
 \draw (2b)--(2c);
 \draw (2b)--(3c);

 \draw (1c)--(1d); 
 \draw (1c)--(2d);
 \draw (1c)--(3d);
 \draw (2c)--(3d);
 \draw (2c)--(4d);
 \draw (3c)--(4d);
 \draw (3c)--(5d);
 \draw (3c)--(6d);

 \draw (1d)--(1e); 
 \draw (1d)--(2e);
 \draw (2d)--(1e);
 \draw (2d)--(3e);
 \draw (3d)--(2e);
 \draw (3d)--(3e);
 \draw (3d)--(4e);
 \draw (4d)--(4e);
 \draw (4d)--(5e);
 \draw (4d)--(6e);
 \draw (5d)--(5e);
 \draw (5d)--(7e);
 \draw (6d)--(6e);
 \draw (6d)--(7e);

 \draw (1e)--(1f); 
 \draw (1e)--(2f);
 \draw (2e)--(2f);
 \draw (2e)--(3f);
 \draw (3e)--(2f);
 \draw (3e)--(4f);
 \draw (4e)--(3f);
 \draw (4e)--(4f);
 \draw (4e)--(5f);
 \draw (4e)--(6f);
 \draw (5e)--(5f);
 \draw (5e)--(7f);
 \draw (6e)--(6f);
 \draw (6e)--(7f);
 \draw (7e)--(7f);
 \draw (7e)--(8f);

 \draw (1f)--(1g); 
 \draw (1f)--(2g);
 \draw (2f)--(2g);
 \draw (2f)--(3g);
 \draw (3f)--(3g);
 \draw (3f)--(4g);
 \draw (3f)--(5g);
 \draw (4f)--(3g);
 \draw (4f)--(6g);
 \draw (4f)--(7g);
 \draw (5f)--(4g);
 \draw (5f)--(6g);
 \draw (5f)--(8g);
 \draw (6f)--(5g);
 \draw (6f)--(7g);
 \draw (6f)--(8g);
 \draw (7f)--(8g);
 \draw (7f)--(9g);
 \draw (8f)--(9g);
 \draw (8f)--(10g);

 \draw (1g)--(1h); 
 \draw (2g)--(1h);
 \draw (2g)--(2h);
 \draw (3g)--(2h);
 \draw (3g)--(3h);
 \draw (3g)--(4h);
 \draw (4g)--(3h);
 \draw (4g)--(5h);
 \draw (5g)--(4h);
 \draw (5g)--(5h);
 \draw (6g)--(3h);
 \draw (6g)--(6h);
 \draw (7g)--(4h);
 \draw (7g)--(6h);
 \draw (8g)--(5h);
 \draw (8g)--(6h);
 \draw (8g)--(7h);
 \draw (9g)--(7h);
 \draw (9g)--(8h);
 \draw (10g)--(8h);

 \draw (1h)--(1i); 
 \draw (2h)--(1i);
 \draw (2h)--(2i);
 \draw (2h)--(3i);
 \draw (3h)--(2i);
 \draw (3h)--(4i);
 \draw (4h)--(3i);
 \draw (4h)--(4i);
 \draw (5h)--(4i);
 \draw (5h)--(5i);
 \draw (6h)--(4i);
 \draw (6h)--(6i);
 \draw (7h)--(5i);
 \draw (7h)--(6i);
 \draw (7h)--(7i);
 \draw (8h)--(7i);

 \draw (1i)--(1j); 
 \draw (1i)--(2j);
 \draw (2i)--(1j);
 \draw (2i)--(3j);
 \draw (3i)--(2j);
 \draw (3i)--(3j);
 \draw (4i)--(3j);
 \draw (4i)--(4j);
 \draw (5i)--(4j);
 \draw (5i)--(5j);
 \draw (6i)--(4j);
 \draw (6i)--(6j);
 \draw (7i)--(5j);
 \draw (7i)--(6j);

 \draw (1j)--(1k); 
 \draw (2j)--(1k);
 \draw (3j)--(1k);
 \draw (3j)--(2k);
 \draw (4j)--(2k);
 \draw (4j)--(3k);
 \draw (5j)--(3k);
 \draw (6j)--(3k);

 \draw (1k)--(1l); 
 \draw (2k)--(1l);
 \draw (2k)--(2l);
 \draw (3k)--(2l);

 \draw (1l)--(1m); 
 \draw (2l)--(1m);

 \end{tikzpicture}

\end{center}

It is easy to observe that $A\in \mathcal{W}_+(6,3)$ and that the fundamental core of $A$ is the following partial map:

$B=\{321|123P, 300|003N, 210|003N, 200|002N, 100|001N, 000|000P \}$.

Hence $\mathcal{S}_B$ is the following $(6,3)$-positively weighted system :

$$
\left\{ \begin{array}{l}
                    x_3 \ge x_2 \ge x_1 \ge 0 > y_1 \ge y_2 \ge y_3  \\
                    x_1 + x_2 + x_3 + y_1 + y_2 + y_3 \ge 0 \\
                    x_3 + y_3 < 0 \\
                    x_2 + y_2 < 0 \\
                    x_1 + y_1 < 0 \\
                    x_2 + x_1 + y_3 < 0.
                    \end{array}
                   \right.
$$

Obviously the previous system $\mathcal{S}_B$ is not compatible, therefore also $\mathcal{S}_A$ is not compatible. Hence $A\in \mathcal{W}_+(6,3)$, but $A \notin \chi(W_+CTSyst(6,3))$.

Let us note that, for the previous map $A$, there does not exist an $f\in W_+F(n,r)$ such that $A=A_f$; therefore this example shows that the answer to the open problems raised in \cite{BisChias} is negative.
\end{example}

The last example tell us that the family $\mathcal{W}_+(n,r)$ does not capture all the properties of the systems in $W_+CTSyst(n,r)$, therefore we give now a more restrictive condition on a family of boolean maps to catch all the properties that characterize a system in $W_+CTSyst(n,r)$.

If $w$ is a string in $S(n,r)$ in the form (\ref{stringa}) with $i_1 \succ \dots \succ i_p \succ 0^§$, $i_{p+1}=\dots i_r= 0^§$ and $j_1 = \dots = j_{q-1} = 0^§$, $0^§ \succ j_q \succ \dots \succ j_{n-r}$, for some indexes $p$ and $q$, we set, see \cite{BisChias}:
\begin{equation*}
w^*=\{ i_1, \cdots, i_p, j_q, \cdots, j_{n-r} \}.
\end{equation*}
For example, if $w=4310|013 \in S(7,4),$ then $w^*=\{ \tilde{1}, \tilde{3}, \tilde{4}, \overline{1}, \overline{3} \}.$ In particular, if $w=0 \cdots 0|0 \cdots 0$ then $w^*=\emptyset.$ \\
It stays defined a bijective map
$$
*: w \in S(n,r) \mapsto w^* \in \mathcal{P} (A(n,r) \setminus \{ 0^§ \} ).
$$
\begin{definition}
If $w\in S(n,r)$, a partition of $w$ is a subset $\{w_1,\cdots,w_k\}$ of $S(n,r)$ such that $\{w_1^*,\cdots,w_k^*\}$ is a set-partition of $w^*$. If $\{w_1,\cdots,w_k\}$ is a partition of $w$ we write $w:w_1 \wr \cdots \wr w_k$.
\end{definition}
\begin{example}
If $w=7543100|0013 \in S(11,7)$,
then $w: 7000000|0000 \wr 5430000|0001 \wr 1000000|0003$.
\end{example}
\begin{definition}
If $A$ is a BPM on $S(n,r)$, we say that $A$ is complemented pointwise if for each $w\in dom(A)$ such that $A(w)=T$, where $T=P$ or $T=N$, and for each partition $w : w_1 \wr\cdots\wr w_k$, with $\{w_1,\cdots,w_k\}\subseteq dom(A)$, we have $A(w_i)=T$ for some $i\in \{1,\cdots,k\}$.
\end{definition}
\begin{definition}
We say that a map $A$ is +formally compatible [-formally compatible] on $S(n,r)$ if :
\begin{itemize}
\item[-] $A\in \mathcal{OP}(n,r)$;\\
\item[-] $A$ is complemented pointwise; \\
\item[-] $A(r\cdots21|12\cdots(n-r)) = P$ [$A(r\cdots21|12\cdots(n-r)) = N$].
\end{itemize}
\end{definition}

We denote by $\mathcal{FC}_+(n,r)$ [$\mathcal{FC}_-(n,r)$] the family of all the maps +formally compatible [-formally compatible] on $S(n,r)$.

It is immediate to observe that $\mathcal{FC}_+(n,r) \subseteq \mathcal{W}_+(n,r)$  [$\mathcal{FC}_-(n,r) \subseteq \mathcal{W}_-(n,r)$] and $\chi(W_+CTSyst(n,r)) \subseteq \mathcal{FC}_+(n,r)$ [$\chi(W_-CTSyst(n,r)) \subseteq \mathcal{FC}_-(n,r)$].

As already indicated in the introduction, we ask :

{\bf open problems :}

Q2) $\chi(W_+CTSyst(n,r))=\mathcal{FC}_+(n,r)$?

Q3) $\chi(W_-CTSyst(n,r))=\mathcal{FC}_-(n,r)$?

The next example shows a map $A \in \mathcal{W}_-(n,r)$ whose $\mathcal{W}_-(n,r)$-fundamental core is complemented pointwise and such that $A$ is not -formally compatible: this proves that the open problem $Q3)$ is false if the system and the boolean map are not total. We can provide an analogue example for the case $A \in \mathcal{W}_+(n,r)$, hence also the problem $Q2)$ is false if the system and the boolean map are not total.

\begin{example}
Let us consider the following map $A\in \mathcal{BT}(6,3)$ (the green nodes are P and the red nodes are N):

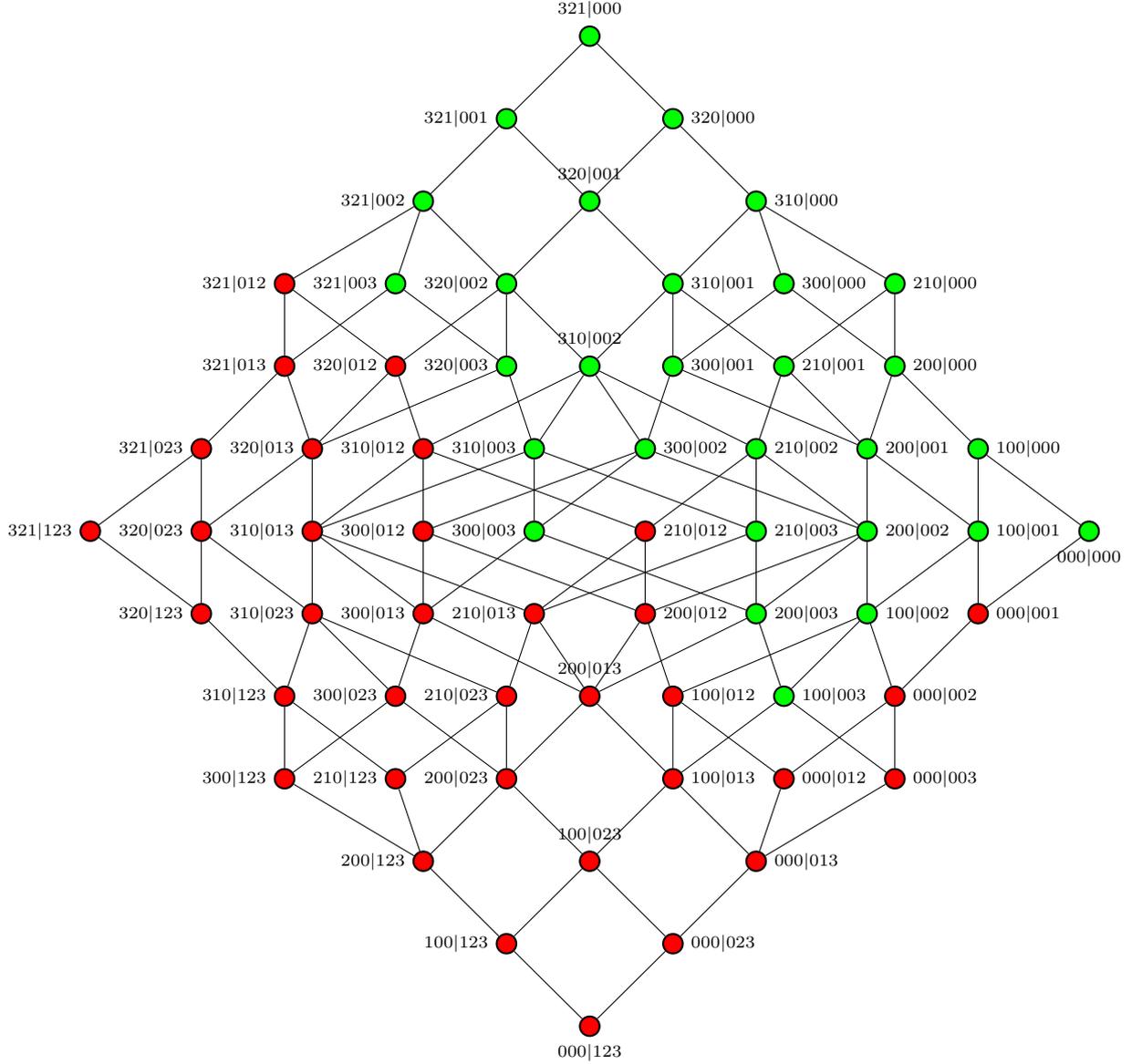
\begin{center}

\begin{tikzpicture}
 [inner sep=1.0mm,
 placeg/.style={circle,draw=black!100,fill=green!100,thick},
 placer/.style={circle,draw=black!100,fill=red!100,thick},
 scale=0.4]

 \path
 (-24,0)node(1a) [placer,label=270:{\footnotesize$000|123$}]{}

 (-27,3)node(1b) [placer,label=180:{\footnotesize$100|123$}]{}
 (-21,3)node(2b) [placer,label=0:{\footnotesize$000|023$}]{}

 (-30,6)node(1c) [placer,label=180:{\footnotesize$200|123$}]{}
 (-24,6)node(2c) [placer,label=90:{\footnotesize$100|023$}]{}
 (-18,6)node(3c) [placer,label=0:{\footnotesize$000|013$}]{}

 (-35,9)node(1d) [placer,label=180:{\footnotesize$300|123$}]{}
 (-31,9)node(2d) [placer,label=180:{\footnotesize$210|123$}]{}
 (-27,9)node(3d) [placer,label=180:{\footnotesize$200|023$}]{}
 (-21,9)node(4d) [placer,label=0:{\footnotesize$100|013$}]{}
 (-17,9)node(5d) [placer,label=0:{\footnotesize$000|012$}]{}
 (-13,9)node(6d) [placer,label=0:{\footnotesize$000|003$}]{}

 (-35,12)node(1e) [placer,label=180:{\footnotesize$310|123$}]{}
 (-31,12)node(2e) [placer,label=180:{\footnotesize$300|023$}]{}
 (-27,12)node(3e) [placer,label=180:{\footnotesize$210|023$}]{}
 (-24,12)node(4e) [placer,label=90:{\footnotesize$200|013$}]{}
 (-21,12)node(5e) [placer,label=0:{\footnotesize$100|012$}]{}
 (-17,12)node(6e) [placeg,label=0:{\footnotesize$100|003$}]{}
 (-13,12)node(7e) [placer,label=0:{\footnotesize$000|002$}]{}

 (-38,15)node(1f) [placer,label=180:{\footnotesize$320|123$}]{}
 (-34,15)node(2f) [placer,label=180:{\footnotesize$310|023$}]{}
 (-30,15)node(3f) [placer,label=180:{\footnotesize$300|013$}]{}
 (-26,15)node(4f) [placer,label=180:{\footnotesize$210|013$}]{}
 (-22,15)node(5f) [placer,label=0:{\footnotesize$200|012$}]{}
 (-18,15)node(6f) [placeg,label=0:{\footnotesize$200|003$}]{}
 (-14,15)node(7f) [placeg,label=0:{\footnotesize$100|002$}]{}
 (-10,15)node(8f) [placer,label=0:{\footnotesize$000|001$}]{}

 (-42,18)node(1g) [placer,label=180:{\footnotesize$321|123$}]{}
 (-38,18)node(2g) [placer,label=180:{\footnotesize$320|023$}]{}
 (-34,18)node(3g) [placer,label=180:{\footnotesize$310|013$}]{}
 (-30,18)node(4g) [placer,label=180:{\footnotesize$300|012$}]{}
 (-26,18)node(5g) [placeg,label=180:{\footnotesize$300|003$}]{}
 (-22,18)node(6g) [placer,label=0:{\footnotesize$210|012$}]{}
 (-18,18)node(7g) [placeg,label=0:{\footnotesize$210|003$}]{}
 (-14,18)node(8g) [placeg,label=0:{\footnotesize$200|002$}]{}
 (-10,18)node(9g) [placeg,label=0:{\footnotesize$100|001$}]{}
 (-6,18)node(10g) [placeg,label=270:{\footnotesize$000|000$}]{}

 (-38,21)node(1h) [placer,label=180:{\footnotesize$321|023$}]{}
 (-34,21)node(2h) [placer,label=180:{\footnotesize$320|013$}]{}
 (-30,21)node(3h) [placer,label=180:{\footnotesize$310|012$}]{}
 (-26,21)node(4h) [placeg,label=180:{\footnotesize$310|003$}]{}
 (-22,21)node(5h) [placeg,label=0:{\footnotesize$300|002$}]{}
 (-18,21)node(6h) [placeg,label=0:{\footnotesize$210|002$}]{}
 (-14,21)node(7h) [placeg,label=0:{\footnotesize$200|001$}]{}
 (-10,21)node(8h) [placeg,label=0:{\footnotesize$100|000$}]{}

 (-35,24)node(1i) [placer,label=180:{\footnotesize$321|013$}]{}
 (-31,24)node(2i) [placer,label=180:{\footnotesize$320|012$}]{}
 (-27,24)node(3i) [placeg,label=180:{\footnotesize$320|003$}]{}
 (-24,24)node(4i) [placeg,label=90:{\footnotesize$310|002$}]{}
 (-21,24)node(5i) [placeg,label=0:{\footnotesize$300|001$}]{}
 (-17,24)node(6i) [placeg,label=0:{\footnotesize$210|001$}]{}
 (-13,24)node(7i) [placeg,label=0:{\footnotesize$200|000$}]{}

 (-35,27)node(1j) [placer,label=180:{\footnotesize$321|012$}]{}
 (-31,27)node(2j) [placeg,label=180:{\footnotesize$321|003$}]{}
 (-27,27)node(3j) [placeg,label=180:{\footnotesize$320|002$}]{}
 (-21,27)node(4j) [placeg,label=0:{\footnotesize$310|001$}]{}
 (-17,27)node(5j) [placeg,label=0:{\footnotesize$300|000$}]{}
 (-13,27)node(6j) [placeg,label=0:{\footnotesize$210|000$}]{}

 (-30,30)node(1k) [placeg,label=180:{\footnotesize$321|002$}]{}
 (-24,30)node(2k) [placeg,label=90:{\footnotesize$320|001$}]{}
 (-18,30)node(3k) [placeg,label=0:{\footnotesize$310|000$}]{}

 (-27,33)node(1l) [placeg,label=180:{\footnotesize$321|001$}]{}
 (-21,33)node(2l) [placeg,label=0:{\footnotesize$320|000$}]{}

 (-24,36)node(1m) [placeg,label=90:{\footnotesize$321|000$}]{};

 \draw (1a)--(1b); 
 \draw (1a)--(2b);

 \draw (1b)--(1c); 
 \draw (1b)--(2c);
 \draw (2b)--(2c);
 \draw (2b)--(3c);

 \draw (1c)--(1d); 
 \draw (1c)--(2d);
 \draw (1c)--(3d);
 \draw (2c)--(3d);
 \draw (2c)--(4d);
 \draw (3c)--(4d);
 \draw (3c)--(5d);
 \draw (3c)--(6d);

 \draw (1d)--(1e); 
 \draw (1d)--(2e);
 \draw (2d)--(1e);
 \draw (2d)--(3e);
 \draw (3d)--(2e);
 \draw (3d)--(3e);
 \draw (3d)--(4e);
 \draw (4d)--(4e);
 \draw (4d)--(5e);
 \draw (4d)--(6e);
 \draw (5d)--(5e);
 \draw (5d)--(7e);
 \draw (6d)--(6e);
 \draw (6d)--(7e);

 \draw (1e)--(1f); 
 \draw (1e)--(2f);
 \draw (2e)--(2f);
 \draw (2e)--(3f);
 \draw (3e)--(2f);
 \draw (3e)--(4f);
 \draw (4e)--(3f);
 \draw (4e)--(4f);
 \draw (4e)--(5f);
 \draw (4e)--(6f);
 \draw (5e)--(5f);
 \draw (5e)--(7f);
 \draw (6e)--(6f);
 \draw (6e)--(7f);
 \draw (7e)--(7f);
 \draw (7e)--(8f);

 \draw (1f)--(1g); 
 \draw (1f)--(2g);
 \draw (2f)--(2g);
 \draw (2f)--(3g);
 \draw (3f)--(3g);
 \draw (3f)--(4g);
 \draw (3f)--(5g);
 \draw (4f)--(3g);
 \draw (4f)--(6g);
 \draw (4f)--(7g);
 \draw (5f)--(4g);
 \draw (5f)--(6g);
 \draw (5f)--(8g);
 \draw (6f)--(5g);
 \draw (6f)--(7g);
 \draw (6f)--(8g);
 \draw (7f)--(8g);
 \draw (7f)--(9g);
 \draw (8f)--(9g);
 \draw (8f)--(10g);

 \draw (1g)--(1h); 
 \draw (2g)--(1h);
 \draw (2g)--(2h);
 \draw (3g)--(2h);
 \draw (3g)--(3h);
 \draw (3g)--(4h);
 \draw (4g)--(3h);
 \draw (4g)--(5h);
 \draw (5g)--(4h);
 \draw (5g)--(5h);
 \draw (6g)--(3h);
 \draw (6g)--(6h);
 \draw (7g)--(4h);
 \draw (7g)--(6h);
 \draw (8g)--(5h);
 \draw (8g)--(6h);
 \draw (8g)--(7h);
 \draw (9g)--(7h);
 \draw (9g)--(8h);
 \draw (10g)--(8h);

 \draw (1h)--(1i); 
 \draw (2h)--(1i);
 \draw (2h)--(2i);
 \draw (2h)--(3i);
 \draw (3h)--(2i);
 \draw (3h)--(4i);
 \draw (4h)--(3i);
 \draw (4h)--(4i);
 \draw (5h)--(4i);
 \draw (5h)--(5i);
 \draw (6h)--(4i);
 \draw (6h)--(6i);
 \draw (7h)--(5i);
 \draw (7h)--(6i);
 \draw (7h)--(7i);
 \draw (8h)--(7i);

 \draw (1i)--(1j); 
 \draw (1i)--(2j);
 \draw (2i)--(1j);
 \draw (2i)--(3j);
 \draw (3i)--(2j);
 \draw (3i)--(3j);
 \draw (4i)--(3j);
 \draw (4i)--(4j);
 \draw (5i)--(4j);
 \draw (5i)--(5j);
 \draw (6i)--(4j);
 \draw (6i)--(6j);
 \draw (7i)--(5j);
 \draw (7i)--(6j);

 \draw (1j)--(1k); 
 \draw (2j)--(1k);
 \draw (3j)--(1k);
 \draw (3j)--(2k);
 \draw (4j)--(2k);
 \draw (4j)--(3k);
 \draw (5j)--(3k);
 \draw (6j)--(3k);

 \draw (1k)--(1l); 
 \draw (2k)--(1l);
 \draw (2k)--(2l);
 \draw (3k)--(2l);

 \draw (1l)--(1m); 
 \draw (2l)--(1m);

 \end{tikzpicture}

\end{center}

It is easy to observe that $A\in \mathcal{W}_-(6,3)$ and that the $\mathcal{W}_-(6,3)$-fundamental core of $A$ is the following partial map:

$B=\{321|012N, 000|001N, 100|003P, 000|000P \}$.

Then $B$ is a BPM on $S(6,3)$ that is complemented pointwise, but $A \notin \mathcal{FC}_-(6,3)$. In fact, if we take the string $w=321|123$, we have $w : 300|003 \wr 100|002 \wr 200|001$, with $A(w)=N$ and $A(300|003)=A(100|002)=A(200|001)=P$.

The system $\mathcal{S}_B$ determined from  $B$ is the following $(6,3)$-weighted system :

$$
\left\{ \begin{array}{l}
                    x_3 \ge x_2 \ge x_1 \ge 0 > y_1 \ge y_2 \ge y_3  \\
                    x_3 + x_2 + x_1 + y_1 + y_2 + y_3 < 0\\
                    x_1 + y_3 \ge 0 \\
          \end{array}
           \right.
$$

The system $\mathcal{S}_B$ is not compatible, because if it were then by n.l.c. also $\mathcal{S}_A$ would be compatible and hence $A \in \chi(W_-CTSyst(6,3)) \subseteq \mathcal{FC}_-(6,3)$, which is a contradiction.

\end{example}

\end{document}